\documentclass[11pt]{amsart}
\usepackage{geometry}                
\usepackage{graphicx}
\usepackage{amssymb}
\usepackage{epstopdf}
\usepackage{geometry} 
\usepackage{graphicx}
\usepackage{amsmath} 
\numberwithin{equation}{section}
\usepackage{commath}
\usepackage{enumitem}
\usepackage{changepage}
\usepackage{cleveref}
\usepackage[normalem]{ulem}
\usepackage{bbm}
\usepackage{subfig}
\usepackage{float}
\usepackage{url}
\usepackage{mathtools}
\usepackage{amsthm} 
\usepackage[disable]{todonotes} 
\DeclareMathOperator*{\esssup}{ess\,sup}

\theoremstyle{plain}
\newtheorem{theorem}{Theorem}[section]
\newtheorem*{theorem*}{Theorem}
\newtheorem{lemma}[theorem]{Lemma}
\newtheorem{corollary}[theorem]{Corollary}
\newtheorem{proposition}[theorem]{Proposition}

\theoremstyle{remark}
\newtheorem*{remark}{Remark}
\theoremstyle{definition}
\newtheorem{definition}[theorem]{Definition}
\allowdisplaybreaks
\title[Stability for the Riemann Problem]{Finite time stability for the Riemann problem with extremal shocks for a large class of hyperbolic systems}
\author[Krupa]{Sam G. Krupa}
\address{Department of Mathematics\\ The University of Texas at Austin\\ Austin, TX 78712\\ USA}
\email{skrupa@math.utexas.edu}
\thanks{This work was partially supported by NSF Grant DMS-1614918.}
\date{May 10th, 2019}                                           

\begin{document}
\begin{abstract}
In this paper on hyperbolic systems of conservation laws in one space dimension, we give a complete picture of stability for all solutions to the Riemann problem which contain only extremal shocks. We study stability of the Riemann problem amongst a large class of solutions. We show stability among the family of solutions with shocks from any family. We assume solutions verify at least one entropy condition.  We have no small data assumptions. The solutions we consider  are bounded and satisfy a strong trace condition weaker than $BV_{\text{loc}}$. We make only mild assumptions on the system. In particular, our work applies to gas dynamics, including the isentropic Euler system and the full Euler system for a polytropic gas. We use the theory of a-contraction (see Kang and Vasseur [{\em Arch. Ration. Mech. Anal.}, 222(1):343--391, 2016]), and introduce new ideas in this direction to allow for two shocks from different shock families to be controlled simultaneously. This paper shows $L^2$ stability for the Riemann problem for all time. Our results compare to Chen, Frid, and Li [{\em Comm. Math. Phys.}, 228(2):201--217, 2002] and Chen and Li [{\em J. Differential Equations}, 202(2):332--353, 2004], which give uniqueness and long-time stability for perturbations of the Riemann problem -- amongst a large class of solutions without smallness assumptions and which are locally $BV$. Although, these results lack global $L^2$ stability.
\end{abstract}
\maketitle
\tableofcontents

\section{Introduction}


We consider the following $n\times n$ system of conservation laws in one space dimension:
\begin{align}
\label{system}
\begin{cases}
\partial_t u + \partial_x f(u)=0,\mbox{ for } x\in\mathbb{R},\mbox{ } t>0,\\
u(x,0)=u^0(x),\mbox{ for } x\in\mathbb{R}.
\end{cases}
\end{align}

For a fixed $T>0$ (including possibly $T=\infty$), the \emph{unknown} is $u\colon\mathbb{R}\times[0,T)\to \mathbb{M}^{n\times 1}$. The function $u^0\colon\mathbb{R}\to\mathbb{M}^{n\times 1}$ is in $L^\infty(\mathbb{R})$ and is the \emph{initial data}. The function $f\colon\mathbb{M}^{n\times 1}\to\mathbb{M}^{n\times 1}$ is the \emph{flux function} for the system. We assume the system \eqref{system} is endowed with a strictly convex entropy $\eta$ and associated entropy flux $q$. Note the system will be hyperbolic on the state space where $\eta$ exists. We assume the functions $f, \eta$, and $q$ are defined on an open convex state space $\mathcal{V}\subset\mathbb{R}^n$. We assume $f,q\in C^2(\mathcal{V})$ and $\eta \in C^2(\mathcal{V})$. By assumption, the entropy $\eta$ and its associated entropy flux $q$ verify the following compatibility relation:
\begin{align}\label{compatibility_relation_eta_system}
\partial_j q =\sum_{i=1}^n \partial_i\eta\partial_j f_i,\hspace{.25in} 1\leq j \leq n.
\end{align}
By convention, the relation \eqref{compatibility_relation_eta_system} is rewritten as 
\begin{align}
\nabla q = \nabla \eta \nabla f,
\end{align}
where $\nabla f$ denotes the matrix $(\partial_j f_i)_{i,j}$.

For $u\in\mathcal{V}$ where $\eta$ exists, the system \eqref{system} is hyperbolic, and the matrix $\nabla f(u)$ is diagonalizable, with eigenvalues 
\begin{align}
\lambda_1(u)\leq \ldots \leq \lambda_n(u),
\end{align}
called \emph{characteristic speeds}.

We consider both bounded  \emph{classical} and bounded \emph{weak} solutions to \eqref{system}. A weak solution $u$ is bounded and measurable and satisfies \eqref{system} in the sense of distributions. I.e., for every Lipschitz continuous test function $\Phi:\mathbb{R}\times[0,T)\to \mathbb{M}^{1\times n}$ with compact support,
\begin{equation}
\begin{aligned}\label{u_solves_equation_integral_formulation_chitchat}
\int\limits_{0}^{T} \int\limits_{-\infty}^{\infty} \Bigg[\partial_t\Phi u + \partial_x\Phi f(u) \Bigg]\,dxdt +\int\limits_{-\infty}^{\infty} \Phi(x,0)u^0(x)\,dx=0.
\end{aligned}
\end{equation}

We only consider solutions $u$ which are entropic for the entropy $\eta$. That is, they satisfy the following entropy condition:
\begin{align}\label{entropy_condition_distributional_system_chitchat}
\partial_t \eta(u)+\partial_x q(u) \leq 0,
\end{align}
in the sense of distributions. I.e., for all positive, Lipschitz continuous test functions $\phi:\mathbb{R}\times[0,T)\to\mathbb{R}$ with compact support:
 \begin{equation}
\begin{aligned}\label{u_entropy_integral_formulation_chitchat}
\int\limits_{0}^{T} \int\limits_{-\infty}^{\infty}\Bigg[\partial_t\phi\big(\eta(u(x,t))\big)+&\partial_x \phi \big(q(u(x,t))\big)\Bigg]\,dxdt+ \int\limits_{-\infty}^{\infty}\phi(x,0)\eta(u^0(x))\,dx\geq0.
\end{aligned}
\end{equation}

For $u_L,u_R\in\mathbb{R}^n$, the function $u:\mathbb{R}\times[0,\infty)\to\mathbb{R}^n$ defined by
\begin{align}\label{shock_solution_system}
u(x,t)\coloneqq
\begin{cases}
u_L &\mbox{ if } x<\sigma t ,\\
u_R &\mbox{ if } x>\sigma t
\end{cases}
\end{align}
is a weak solution to \eqref{system} if and only if  $u_L,u_R$, and $\sigma$ satisfy the Rankine-Hugoniot jump compatibility relation:
\begin{align}\label{RH_jump_condition}
f(u_R)-f(u_L)=\sigma (u_R-u_L),
\end{align}
in which case \eqref{shock_solution_system} is called a \emph{shock} solution.

Moreover, the solution \eqref{shock_solution_system} will be entropic for $\eta$ (according to \eqref{u_entropy_integral_formulation_chitchat})  if and only if,
\begin{align}\label{entropic_shock_condition_system}
q(u_R)-q(u_L)\leq \sigma (\eta(u_R)-\eta(u_L)).
\end{align}
In this case, $(u_L,u_R,\sigma)$ is an \emph{entropic Rankine--Hugoniot discontinuity}.

For a fixed $u_L$, we consider the set of $u_R$ which satisfy \eqref{RH_jump_condition} and \eqref{entropic_shock_condition_system} for some $\sigma$. For a general $n\times n$ strictly hyperbolic system of conservation laws endowed with a strictly convex entropy , we know that locally this set of $u_R$ values is made up of $n$ curves (see for example \cite[p.~140-6]{lefloch_book}). 

The present paper concerns the finite-time stability of Riemann problem solutions to \eqref{system}, working in the $L^2$ setting. We work in a very general setting. Our techniques are based on the theory of shifts in the context of the relative entropy method as developed by Vasseur (see \cite{VASSEUR2008323}). We consider systems of the form \eqref{system}, with minimal assumptions on the shock families. We ask that the extremal shock speeds (1-shock and n-shock speeds) are separated from the intermediate shock families. If we want to consider solutions to the Riemann problem with a 1-shock, we ask that the 1-shock family satisfy the Liu entropy condition (shock speed decreases as the right-hand state travels down the 1-shock curve), and we ask that the shock strength increase in the sense of relative entropy (an $L^2$ requirement) as the right-hand state travels down the 1-shock curve. If we want to consider n-shocks, we ask for similar requirements on the n-shock family.

 The intermediate wave families have far fewer requirements. The intermediate shock curves might not even be well-defined and characteristic speeds might cross.
 
Systems we have in mind include the isentropic Euler system and the  full Euler system for a polytropic gas (both in Eulerian coordinates). 

We study the stability of solutions $\bar{v}$ to the Riemann problem. We study the stability and uniqueness of these solutions among a large class of weak solutions $u$ which are bounded, measurable, entropic for at least one strictly convex entropy, and verify a strong trace condition (weaker than $BV_{\text{loc}}$). We require the solution  $\bar{v}$ contain shocks of only the extremal families (1-shocks and n-shocks), if it contains shocks at all. However, the rougher solutions $u$ which we compare to $\bar{v}$  may have shocks of any type or family.

Previous results in this direction include Chen, Frid, and Li \cite{MR1911734} where for the full Euler system,  they show uniqueness and long-time stability for perturbations of Riemann initial data among a large class of entropy solutions (locally $BV$ and without smallness conditions) for the $3\times3$ Euler system in Lagrangian coordinates. They also show uniqueness for solutions piecewise-Lipschitz in $x$. For an extension to the relativistic Euler equations, see Chen and Li \cite{MR2068444}. However, these papers do not give $L^2$ stability results for all time.

We will occasionally use a strong form of Lax's E-condition, saying we want a shock to be compressive but not overcompressive \cite[p.~359-60]{dafermos_big_book},

\begin{align}
\begin{cases}\label{lax_strong}
\mbox{for a shock with left state $u_L$, right state $u_R$, and shock speed $\sigma$,}
\\
\mbox{there exists $i\in\{1,\ldots,n\}$ such that }\\
\lambda_{i-1}(u_{L/R})<\lambda_i(u_R)\leq \sigma\leq\lambda_i(u_L)<\lambda_{i+1}(u_{L/R}),\\
\mbox{where $u_{L/R}$ denotes $u_L$ or $u_R$ and $\lambda_0\coloneqq -\infty$ and $\lambda_{n+1}\coloneqq\infty$ }.
\end{cases}
\end{align}

The condition \eqref{lax_strong} is a type of separation between the characteristic speeds. In particular, note that for any strictly hyperbolic system of conservation laws, the first and last inequalities in \eqref{lax_strong} will hold whenever $\abs{u_L-u_R}$ is sufficiently small. Furthermore, for hyperbolic systems where the characteristic speeds are completely separated in value, any shock $(u_L,u_R)$ will trivially satisfy \eqref{lax_strong}. For example, for the full Euler system for gas dynamics in Lagrangian coordinates the first characteristic speed is always negative, the middle one is always zero, and the third characteristic speed is always positive.

We will also occasionally consider systems of the form \eqref{system} (endowed with the entropy $\eta$) and verifying the additional  sign condition,
\begin{align}
\begin{cases}\label{RHS_good_sign_rarefaction}
\Bigg(\partial_x \bigg|_{(x,t)}\hspace{-.21in} \bar{v}^T(x,t)\Bigg)\nabla^2\eta(\bar{v}(x,t)) f(u|\bar{v}(x,t))\geq0,\\
\vspace{-.08in}
\\
\mbox{for every rarefaction wave solution $\bar{v}$ of \eqref{system} and for every $u\in\mathbb{R}^n$,}
\end{cases}
\end{align}

where $f(\cdot|\cdot)$ denotes the relative flux,
\begin{align*}
f(a|b)\coloneqq f(a)-f(b)-\nabla f (b)(a-b),
\end{align*}
for $a,b\in\mathbb{M}^{n\times 1}$.

In particular, the system of isentropic gas dynamics verifies the property \eqref{RHS_good_sign_rarefaction}. For a proof of this fact, see \cite{VASSEUR2008323}. The full Euler system also satisfies \eqref{RHS_good_sign_rarefaction} in the case of one space dimension and multiple space dimensions (see \cite{MR3357629} for proof of this  in multiple space dimensions).

Fix $T>0$. For $u_L,u_R\in\mathbb{R}^n$, assume there exists a (potentially weak) solution $\bar{v}\in L^\infty(\mathbb{R}\times[0,T))$ entropic for the entropy $\eta$, with the initial data
\begin{align}\label{Riemann_problem}
\bar{v}(x,0)=
\begin{cases}
u_L & \text{if } x<0\\ 
u_R & \text{if } x>0.
\end{cases}
\end{align}
In other words, $\bar{v}$ solves the Riemann problem \eqref{Riemann_problem}.

Assume $\bar{v}$ has the following standard form for a solution to the Riemann problem, constant on lines through the origin in the x-t plane: 
\begin{align}\label{standard_form_solution_RP}
\begin{cases}
\text{$\bar{v}$ is made up of $n+1$ constant states  $u_L=\bar{v}_1,\ldots,\bar{v}_{n+1}=u_R$, where if $\bar{v}_i\neq\bar{v}_{i+1}$,}
\\
\text{then $\bar{v}_i$ is joined to $\bar{v}_{i+1}$ by either an i-shock or an i-rarefaction fan.}
\\
\text{Otherwise $\bar{v}_i=\bar{v}_{i+1}$ and we do not need a shock or a rarefaction to connect $\bar{v}_i$ to $\bar{v}_{i+1}$.}
\end{cases}
\end{align}

Before we can present our stability and uniqueness results, for a fixed $\bar{v}$ as in \eqref{standard_form_solution_RP}, we define the Property $(\mathcal{D})$. 

We say a function $\Psi_{\bar{v}}\colon\mathbb{R}\times[0,T)\to\mathbb{R}^n$ verifies property  $(\mathcal{D})$ if

\vspace{.15in}

$(\mathcal{D})$:
\begin{itemize}
\item
If $\bar{v}$ contains at least one rarefaction wave, and if there are any shocks in $\bar{v}$ they are either a 1-shock verifying  \eqref{lax_strong} or an n-shock verifying \eqref{lax_strong}, then
\begin{itemize}
\item
If $\bar{v}$ contains a 1-shock verifying \eqref{lax_strong} for $i=1$, but no other shocks, then there exists a Lipschitz continuous function $h_1:[0,T)\to\mathbb{R}$ with $h_1(0)=0$ and verifying
\begin{align}\label{h_verify_1}
h_1(t)<\lambda_2(\bar{v}_2)t
\end{align}
for all $t\in[0,T)$ such that $\Psi_{\bar{v}}:\mathbb{R}\times[0,T)\to\mathbb{R}^n$ verifies,
\begin{align}
\Psi_{\bar{v}}(x,t)\coloneqq
\begin{cases}\label{Psi_def_RP_one_shock_theorem}
\bar{v}_1& \text{if } x<h_1(t)\\
\bar{v}_2& \text{if } h_1(t)<x<\lambda_2(\bar{v}_2)t\\
\bar{v}(x,t)& \text{if } \lambda_2(\bar{v}_2)t<x.
\end{cases}
\end{align}
\item
If $\bar{v}$ contains an n-shock verifying \eqref{lax_strong} for $i=n$, but no other shocks, then there exists a  Lipschitz continuous function $h_n:[0,T)\to\mathbb{R}$ with $h_n(0)=0$, and verifying
\begin{align}\label{h_verify_2}
\lambda_{n-1}(\bar{v}_n)t<h_n(t)
\end{align}
for all $t\in[0,T)$ such that $\Psi_{\bar{v}}:\mathbb{R}\times[0,T)\to\mathbb{R}^n$ verifies,
\begin{align}
\Psi_{\bar{v}}(x,t)\coloneqq
\begin{cases}\label{Psi_def_RP_n_shock_theorem}
\bar{v}(x,t)& \text{if } x<\lambda_{n-1}(\bar{v}_n)t\\
\bar{v}_n& \text{if } \lambda_{n-1}(\bar{v}_n)t<x<h_n(t)\\
\bar{v}_{n+1}& \text{if } h_n(t)<x.
\end{cases}
\end{align}
\item
If $\bar{v}$ contains a 1-shock and an n-shock verifying \eqref{lax_strong} for $i=1$ and $i=n$, respectively, but no other shocks, then there exists Lipschitz continuous functions $h_1,h_n:[0,T)\to\mathbb{R}$ with $h_1(0)=h_n(0)=0$, where $h_1$ verifies \eqref{h_verify_1} and $h_n$ verifies \eqref{h_verify_2} such that $\Psi_{\bar{v}}:\mathbb{R}\times[0,T)\to\mathbb{R}^n$ verifies,
\begin{align}
\Psi_{\bar{v}}(x,t)\coloneqq
\begin{cases}\label{Psi_def_RP_one_shock_n_shock_theorem}
\bar{v}_1& \text{if } x<h_1(t)\\
\bar{v}_2& \text{if } h_1(t)<x<\lambda_2(\bar{v}_2)t\\
\bar{v}(x,t)& \text{if } \lambda_2(\bar{v}_2)t<x<\lambda_{n-1}(\bar{v}_n)t\\
\bar{v}_n& \text{if } \lambda_{n-1}(\bar{v}_n)t<x<h_n(t)\\
\bar{v}_{n+1}& \text{if } h_n(t)<x.
\end{cases}
\end{align}
\item
If $\bar{v}$ contains no shocks, then $\Psi_{\bar{v}}:\mathbb{R}\times[0,T)\to\mathbb{R}^n$ verifies,
\begin{align}
\Psi_{\bar{v}}(x,t)\coloneqq \bar{v}(x,t)
\end{align}
for all $(x,t)\in \mathbb{R}\times[0,T)$.
\end{itemize}
\item If $\bar{v}$ \emph{does not contain any rarefactions}, and if $\bar{v}$ contains any shocks, they are either a 1-shock or an n-shock, then
\begin{itemize}
\item
If $\bar{v}$ contains a 1-shock, but no other shocks, then there exists a Lipschitz continuous function $h_1:[0,T)\to\mathbb{R}$ with $h_1(0)=0$ such that $\Psi_{\bar{v}}:\mathbb{R}\times[0,T)\to\mathbb{R}^n$ verifies,
\begin{align}
\Psi_{\bar{v}}(x,t)\coloneqq
\begin{cases}\label{Psi_def_RP_one_shock_theorem}
\bar{v}_1& \text{if } x<h_1(t)\\
\bar{v}_2& \text{if } h_1(t)<x.
\end{cases}
\end{align}
\item
If $\bar{v}$ contains an n-shock, but no other shocks, then there exists a Lipschitz continuous function $h_n:[0,T)\to\mathbb{R}$ with $h_n(0)=0$ such that $\Psi_{\bar{v}}:\mathbb{R}\times[0,T)\to\mathbb{R}^n$ verifies,
\begin{align}
\Psi_{\bar{v}}(x,t)\coloneqq
\begin{cases}\label{Psi_def_RP_n_shock_theorem}
\bar{v}_n& \text{if }x<h_n(t)\\
\bar{v}_{n+1}& \text{if } h_n(t)<x.
\end{cases}
\end{align}
\item
If $\bar{v}$ contains a 1-shock and an n-shock, but no other shocks, then there exists Lipschitz continuous functions $h_1,h_n:[0,T)\to\mathbb{R}$ with $h_1(0)=h_n(0)=0$, and verifying
\begin{align}\label{h_verify_3}
h_1(t)\leq h_n(t)
\end{align}
for all $t\in[0,T)$ such that $\Psi_{\bar{v}}:\mathbb{R}\times[0,T)\to\mathbb{R}^n$ verifies,
\begin{align}
\Psi_{\bar{v}}(x,t)\coloneqq
\begin{cases}\label{Psi_def_RP_one_shock_n_shock_theorem}
\bar{v}_1& \text{if } x<h_1(t)\\
\bar{v}_2& \text{if } h_1(t)<x<h_n(t)\\
\bar{v}_{n+1}& \text{if } h_n(t)<x.
\end{cases}
\end{align}
\item
If $\bar{v}$ contains no shocks, then $\bar{v}$ must be a constant function and  $\Psi_{\bar{v}}:\mathbb{R}\times[0,T)\to\mathbb{R}^n$ verifies,
\begin{align}
\Psi_{\bar{v}}(x,t)\coloneqq \bar{v}
\end{align}
for all $(x,t)\in \mathbb{R}\times[0,T)$.
\end{itemize}
\end{itemize}

Let $u\in L^\infty(\mathbb{R}\times[0,T))$ be \emph{any weak solution} to \eqref{system}, entropic for the entropy $\eta$ (assume also that $u$ has strong traces (\Cref{strong_trace_definition})). With the definition of Property $(\mathcal{D})$ out of the way, we present our main and most important theorem regarding $L^2$-type stability and uniqueness results between $u$ and $\bar{v}$.  The hypotheses $(\mathcal{H})$ and $(\mathcal{H})^*$ in the theorem  depend only on the system \eqref{system} and the Riemann problem solution $\bar{v}$. The hypotheses are related to conditions on 1-shocks and n-shocks and in particular are satisfied by the isentropic Euler and full Euler systems. They are with small modifications related to the hypotheses in \cite{Leger2011}. These hypotheses are explained in detail in \Cref{hypotheses_on_system}. The theorem gives a general overview of the results in this paper:

\begin{theorem}[Main theorem -- $L^2$ Stability for the Riemann Problem with Extremal Shocks]
\label{RP_theorem_main_theorem}
Fix $T>0$. Assume $u,\bar{v}\in L^\infty(\mathbb{R}\times[0,T))$ are solutions to the system \eqref{system}. Assume that $u$ and $\bar{v}$ are entropic for the entropy $\eta$. Further, assume that $u$ has strong traces (\Cref{strong_trace_definition}). 

Assume also that $\bar{v}$ is a solution to the Riemann problem \eqref{Riemann_problem} and has the form \eqref{standard_form_solution_RP}. If $\bar{v}$ contains a 1-shock, assume the hypotheses $(\mathcal{H})$ hold. Likewise, if $\bar{v}$ contains an n-shock, assume the hypotheses $(\mathcal{H})^*$ hold. 

Assume \eqref{RHS_good_sign_rarefaction} holds. Further, assume the system \eqref{system} has at least two conserved quantities ($n\geq2$).
 
If  $\bar{v}$ \emph{contains at least one rarefaction wave}, assume that if there are any shocks in $\bar{v}$ they are either a 1-shock verifying  \eqref{lax_strong} or an n-shock verifying \eqref{lax_strong}.

If $\bar{v}$ does not contain any rarefactions, and if $\bar{v}$ contains any shocks, assume they are either a 1-shock or an n-shock (and we do not require \eqref{lax_strong}).

\vspace{.15in}

Then there exists a $\Psi_{\bar{v}}$ with Property $(\mathcal{D})$, and verifying the following stability estimate:
\begin{align}\label{RP_lax_strong_stability_theorem_result_Main_theorem}
\int\limits_{-R}^{R}\abs{u(x,t_0)-\Psi_{\bar{v}}(x,t_0)}^2\,dx\leq\mu\int\limits_{-R-rt_0}^{R+rt_0}\abs{u^0(x)-\bar{v}(x,0)}^2\,dx,
\end{align}
for a constant $\mu>0$, and for all $t_0,R>0$ verifying $t_0\in(0,R)$ and 
\begin{align}\label{R_sufficiently_large_RP_theorem_statement}
R>\max_{i}\{\mbox{Lip}[h_i]\}t_0,
\end{align}
where the max runs over the i-shock families contained in $\bar{v}$ (1-shocks and/or n-shocks) and the $h_i$ are in the context of Property $(\mathcal{D})$.

We also have the following $L^2$-type control on the shift functions $h_i$:
\begin{align}\label{control_shifts_two_shifts204}
\int\limits_0^{t_0} \sum_{i}\abs{\sigma^i(\bar{v}_i,\bar{v}_{i+1})-\dot{h}_i(t)}^2\,dt\leq \mu\int\limits_{-R-rt_0}^{R+rt_0}\abs{u^0(x)-\bar{v}(x,0)}^2\,dx,
\end{align}
where the sum runs over the i-shock families contained in $\bar{v}$ (1-shocks and/or n-shocks). 
\end{theorem}

\begin{remark}

\hfill

\begin{itemize}
\item
Note that H\"older's inequality and \eqref{control_shifts_two_shifts204} give control on the shifts in the form of 
\begin{align}
\frac{1}{t_0}\int\limits_0^{t_0}\abs{\sigma^i(\bar{v}_i,\bar{v}_{i+1})-\dot{h}_i(t)}\,dt \leq \frac{\sqrt{\mu}}{\sqrt{t_0}}\norm{u^0(\cdot)-\bar{v}(\cdot,0)}_{L^2(-R-rt_0,R+rt_0)}.
\end{align}
\item The relative entropy method can handle the occurance of vacuum states in the weak, entropic solution $u$ (where $u$ is in the context of \Cref{RP_theorem_main_theorem}). In particular, the method of relative entropy can be extended to allow for vacuum states in the  first slot of the relative entropy $\eta(\hspace{.005in}\cdot\hspace{.005in}|\hspace{.005in}\cdot\hspace{.005in})$. For simplicity, in the present article we do not consider these generalizations to vacuum states. However, our results and arguments would be the same even if we considered vacuum states in the solution $u$. For details, see \cite{VASSEUR2008323}, \cite[p.~346-7]{MR3519973}, and \cite[p.~277-8]{Leger2011}.
\end{itemize}
\end{remark}

For more details on the results in \Cref{RP_theorem_main_theorem}, see \Cref{RP_theorem_1} and \Cref{RP_theorem_3} below. 

\vspace{.07in}

Our method is the relative entropy method, a technique created by Dafermos \cite{doi:10.1080/01495737908962394,MR546634} and DiPerna \cite{MR523630}  to give $L^2$-type stability estimates between a Lipschitz continuous solution and a rougher solution, which is only weak and entropic for a strictly convex entropy (the so-called \emph{weak-strong} stability theory). For a system \eqref{system} endowed with an entropy $\eta$, the technique of relative entropy considers the quantity called the
\emph{relative entropy}, defined as
\begin{align}
\eta(u|v)\coloneqq \eta(u)-\eta(v)-\nabla\eta(v)\cdot (u-v).
\end{align}

Similarly, we define relative entropy-flux,
\begin{align}
q(u;v)\coloneqq q(u)-q(v)-\nabla\eta(v)\cdot (f(u)-f(v)).
\end{align}

Remark that for any constant $v\in\mathbb{R}^n$, the map $u\mapsto\eta(u|v)$ is an entropy for the system \eqref{system}, with associated entropy flux $u\mapsto q(u;v)$. Furthermore, if $u$ is a weak solution to \eqref{system} and entropic for $\eta$, then $u$ will also be entropic for $\eta(\cdot|v)$. This can be calculated directly from \eqref{system} and \eqref{entropy_condition_distributional_system_chitchat} -- note that the map $u\mapsto\eta(u|v)$ is basically $\eta$ plus a linear term.

Moreover, by virtue of $\eta$ being \emph{strictly} convex, the relative entropy is comparable to the $L^2$ distance, in the following sense:

\begin{lemma}\label{entropy_relative_L2_control_system} For any fixed compact set $V\subset\mathcal{V}$, there exists  $c^*,c^{**}>0$ such that for all $u,v\in V$,
\begin{align}
c^*\abs{a-b}^2\leq \eta(u|v)\leq c^{**}\abs{a-b}^2.
\end{align}
The constants $c^*,c^{**}$ depend on $V$ and bounds on the second derivative of $\eta$.
\end{lemma}

This lemma follows from Taylor's theorem; for a proof see \cite{Leger2011,VASSEUR2008323}.

Now that we have defined the relative entropy, we remark that what we prove in this article is actually stronger than \Cref{RP_theorem_main_theorem}. We get more than the $L^2$ stability estimate \eqref{RP_lax_strong_stability_theorem_result_Main_theorem}. In fact, \emph{we get a contraction in a properly defined pseudo-distance}. For simplicity, here in the introduction we define the pseudo-distance only when $\bar{v}$ (in the context of  \Cref{RP_theorem_main_theorem}) contains two shocks. The definition of the pseudo-distance is very similar for the case of one shock or no shock. Then: for $u$, $\bar{v}$, $\Psi_{\bar{v}}$ as in the context of \Cref{RP_theorem_main_theorem} and $\alpha,\beta>0$, we define the pseudo-distance
\begin{equation}
\begin{aligned}\label{pseudo_distance_Kang}
&E\Big(u(\cdot,t);\Psi_{\bar{v}}(\cdot,t);\alpha;\beta\Big)\coloneqq
\Bigg[\int\limits_{-R}^{h_1(t)}\eta(u(x,t)|\Psi_{\bar{v}}(x,t))\,dx+\alpha\int\limits_{h_1(t)}^{h_n(t)}\eta(u(x,t)|\Psi_{\bar{v}}(x,t))\,dx\\
&\hspace{3.3in}+\beta\int\limits_{h_n(t)}^{R}\eta(u(x,t)|\Psi_{\bar{v}}(x,t))\,dx\Bigg],
\end{aligned} 
\end{equation}
and where $R>0$ is just a large constant which allows us to consider the solution $u$ only locally. The $h_1$ and $h_n$ used in the definition \eqref{pseudo_distance_Kang} are from the Property $(\mathcal{D})$ which $\Psi_{\bar{v}}$ verifies. 

The pseudo-distance \eqref{pseudo_distance_Kang} is a technical tool we use in the proof of \Cref{RP_theorem_main_theorem} (and in particular \Cref{RP_theorem_1} and \Cref{RP_theorem_3}). By \Cref{entropy_relative_L2_control_system}, it gives us the $L^2$ stability estimate \eqref{RP_lax_strong_stability_theorem_result_Main_theorem}. The constants $\alpha$ and $\beta$ we choose do not depend on the weak, entropic solution $u$. Our use of the pseudo-distance \eqref{pseudo_distance_Kang} is based on the work \cite{MR3519973}.

\vspace{.07in}

Given a Lipschitz continuous solution $\bar{u}$ to \eqref{system}, and weak solution $u$ to \eqref{system} which is entropic for at least one entropy, the method of relative entropy can be used to determine estimates on the growth in time of 
\begin{align}\label{example_l2_control222}
\norm{\bar{u}(\cdot,t)-u(\cdot,t)}_{L^2(\mathbb{R})}.
\end{align}
To estimate the growth of this quantity, consider $\partial_t\int\eta(u|\bar{u})\,dx$. By \eqref{entropy_relative_L2_control_system}, we get estimates of the $L^2$-type \eqref{example_l2_control222}. The point is that due the entropy inequality \eqref{entropy_condition_distributional_system_chitchat}, it is more natural to consider the quantity $\int\eta(u|\bar{u})\,dx$ than to consider the $L^2$ norm itself.

However, the relative entropy method breaks down if a discontinuity is introduced into the otherwise smooth solution $\bar{u}$. In fact, simple examples for the scalar conservation laws show that when $\bar{u}$ has a discontinuity, there is no $L^2$ stability in the same sense as in the classical weak-strong estimates.

In order to recover $L^2$ stability in the sense of the classical weak-strong estimates, we must allow the discontinuity in $\bar{u}$ to be moved (`shifted') with an artificial velocity which depends on the weak solution $u$. This is the theory of shifts. Within the context of the relative entropy method, this idea was devised by Vasseur \cite{VASSEUR2008323}. Since then, this technique has been the subject of intense study by Vasseur and his team, and has yielded new results. The first result was for the scalar conservation laws in one space dimension. Further work considered the scalar viscous conservation laws in one space dimension \cite{MR3592682} and multiple space dimensions \cite{multi_d_scalar_viscous_9122017}. To handle systems, which allow for shocks from differing wave families, the technique of a-contraction is used \cite{MR3519973,MR3479527,MR3537479,serre_vasseur,Leger2011}. Recent work for scalar \cite{2017arXiv170905610K} has also allowed for many discontinuities to exist in the otherwise classical solution $\bar{u}$ which the method of relative entropy considers. By adding more and more discontinuities to the otherwise classical solution $\bar{u}$, the method of relative entropy and the theory of shifts can be used to show uniqueness for solutions which are entropic for at least one strictly convex entropy.  For a general overview of theory of shifts and the relative entropy method, see \cite[Section 3-5]{MR3475284}. 
The theory of stability up to a shift has also been used to study the asymptotic limits when the limit is discontinuous (see \cite{MR3333670} for the scalar case, \cite{MR3421617} for systems). There are many other results using the relative entropy method to study the asymptotic limit. However, without the theory of shifts these results can only consider limits which are Lipschitz continuous (see  \cite{MR1842343,MR1980855,MR2505730,MR1115587,MR1121850,MR1213991,MR2178222,MR2025302} and \cite{VASSEUR2008323} for a survey).
 
The present paper is another step in this program of stability up to a shift.

We use the construction of shifts based on the generalized characteristic introduced in \cite{scalar_move_entire_solution,move_entire_solution_system}. In this paper, we are able to handle shocks from two different wave families in the same solution, which is necessary for handling the Riemann problem with shocks from extremal wave families.
 As mentioned in \cite{scalar_move_entire_solution,move_entire_solution_system}, the generalized-characteristic-based shifts are an improvement over previous shift constructions partly because they are very simple, and thus amenable to analysis and control. In particular, to do stability estimates for a solution to the Riemann problem with two extremal shocks, we need two shifts --  one for each shock. Using prior constructions of the shifts, it was impossible to tell if the two shifts necessary for the Riemann problem would interact in a bad way or not. Using generalized-characteristic-based shifts, this analysis is easy: due to the separation of shock speeds, and the fact that generalized-characteristic-based shifts travel at  characteristic-like speed (for the characteristic of the shock they are shifting), we know immediately that the shifts for a 1-shock will stay to the left of the shifts for an n-shock.  See \Cref{RP_theorem_3} and \Cref{systems_entropy_dissipation_room}.
 
In order to control the two shifts, one for the 1-shock and one for the n-shock in a solution to the Riemann problem, we needed to extend the theory of Filippov flows to construct  the two shifts in the sense of Filippov flows, while  still maintaining control on their ordering (keeping the 1-shock shift to the left of the n-shock shift). The need to control the ordering of two different Filippov flows arose in the first paper where the theory of shifts in the context of the relative entropy method was used (see \cite[Proposition 2]{Leger2011_original}). However, our result (\Cref{Filippov_existence_RP}) is more general and has a significantly simpler proof.

In \cite{move_entire_solution_system}, for a solution $\bar{u}$ to \eqref{system} which is Lipschitz continuous on both sides of one single shock curve in space-time, to maintain $L^2$ stability between $\bar{u}$ and another solution $u$ which is weak and entropic for at least one entropy, the solution $\bar{u}$ is translated artificially in space, instead of simply moving only the discontinuity itself. However, if $\bar{u}$ is a solution to the Riemann problem, it might contain rarefactions, which have a blow up in the derivative at $t=0$. This causes tremendous entropy production if the rarefaction is artificially translated in space. Moreover, $\bar{u}$ could easily contain two shocks -- making it impossible to artificially translate $\bar{u}$ in such a way that each discontinuity is moving at the velocity necessary to maintain $L^2$ stability against the solution $u$. Both of these concerns, $\bar{u}$ containing two shocks and the blowup  of rarefactions at $t=0$, are addressed in the present paper. See \Cref{section_overview} for a related discussion.

For hyperbolic systems of conservation laws in one space dimension, one difficulty to showing stability and uniqueness of (entropic) solutions is that many systems admit only a single nontrivial entropy. The best well-posedness theory to date has been the $L^1$-based theory of Bressan, Crasta, and Piccoli  \cite{MR1686652}. However, this work only considers solutions with small total variation. It would be interesting to study the stability of these solutions in a larger class. In fact, existence of solutions to the $2\times 2$ Euler system is known. 

The present paper is a step towards a better understanding of the well-posedness of hyperbolic systems of conservation laws in one space dimension. Our techniques are of $L^2$-type. We use the relative entropy method and the related theories of shifts and a-contraction. Due to these theories not being perturbative, we are able to prove results without small data limitations. Furthermore, because we use techniques based on the relative entropy method, we only use a single entropy and require only a single entropy condition. 

\vspace{.07in}

The outline of the paper is as follows: in \Cref{hypotheses_on_system}, we give our hypotheses on the system. In \Cref{section_overview}, we present an overview of the proof of the main theorem (\Cref{RP_theorem_main_theorem}), which is actually proved in two parts: \Cref{RP_theorem_1} and \Cref{RP_theorem_3}.  In \Cref{technical_lemmas}, we present technical lemmas. In \Cref{construction_of_the_shift}, we construct the shift. Finally, in \Cref{proofs_of_multiple_stability_theorems} we prove \Cref{RP_theorem_1} and \Cref{RP_theorem_3}, which make up the main theorem \Cref{RP_theorem_main_theorem}.

\section{Hypotheses on the system}\label{hypotheses_on_system}
We will consider the following structural hypotheses $(\mathcal{H})$, $(\mathcal{H})^*$ on the system \eqref{system}, \eqref{entropy_condition_distributional_system_chitchat} regarding the 1-shock and n-shock curves (they are closely related to hypotheses in \cite{Leger2011,move_entire_solution_system,MR3519973}). For a fixed i-shock $(v_L,v_R)$ ($i=1$ or $i=n$):
\hfill \break

\begin{itemize}
\item
$(\mathcal{H}1)$: (Family of 1-shocks verifying the Liu condition) There exists $r_0>0$ such that for all $u\in B_{r_0}(v_L)$, there is a 1-shock curve (issuing from $u$) $S_u^1\colon[0,s_u)\to \mathcal{V}$ (possibly $s_u=\infty$) parameterized by arc length. Moreover, $S_u^1(0)=u$ and the Rankine-Hugoniot jump condition holds:
\begin{align}
f(S_u^1(s))-f(u)=\sigma^1_u(s)(S_u^1(s)-u),
\end{align}
where $\sigma^1_u(s)$ is the velocity function. The map $u\mapsto s_u$ is Lipschitz on $\mathcal{V}$. Further, the maps $(s,u)\mapsto S_u^1(s)$ and $(s,u)\mapsto \sigma^1_u(s)$ are both $C^1$ on $\{(s,u)|s\in [0,s_u), u\in\mathcal{V}\}$, and the following conditions are satisfied: 
\begin{align*}
&\mbox{(a) (Liu entropy condition) } \frac{\mbox{d}}{\mbox{d}s} \sigma^1_u(s) <0,\hspace{.2in} \sigma^1_u(0)=\lambda_1(u), 
\\&\mbox{(b) (shock ``strengthens'' with $s$) }  \frac{\mbox{d}}{\mbox{d}s}\eta(u|S_u^1(s))>0, \hspace{.2in}\mbox{for all } s>0,
\\&\mbox{(c) (the shock curve cannot wrap tightly around itself)}
\\&\hspace{.5in}\mbox{For all $R>0$, there exists $\tilde{S}>0$ such that}  
\\
&\hspace{.8in}\Big\{S^1_{u}(s) \Big| s\in[0.s_u), \abs{u}\leq R \mbox{ and } \abs{S^1_u(s)}\leq R\Big\} \subseteq \Big\{S^1_u(s) \Big| \abs{u} \leq R \mbox{ and } s\leq \tilde{S}\Big\}.
\end{align*}

\item
$(\mathcal{H}2)$: If $(u_L,u_R)$ is an entropic Rankine-Hugoniot discontinuity with shock speed $\sigma$, then $\sigma> \lambda_1(u_R)$.

\item
$(\mathcal{H}3)$: If $(u_L,u_R)$ (with $u_L\in B_{r_0}(v_L)$) is an entropic Rankine-Hugoniot discontinuity with shock speed $\sigma$ verifying
\begin{align}
\sigma\leq\lambda_1(u_L),
\end{align}
then $u_R$ is in the image of $S_{u_L}^1$. In other words, there exists $s_{u_R}\in[0,s_{u_L})$ such that $S_{u_L}^1(s_{u_R})=u_R$ (and by implication, $\sigma=\sigma^1_{u_L}(s_{u_R})$).
\end{itemize}

Similarly, we will consider the following structural hypotheses $(\mathcal{H})^*$ on the system \eqref{system}, \eqref{entropy_condition_distributional_system_chitchat} regarding the n-shock curves:
\hfill \break

\begin{itemize}
\item
$(\mathcal{H}1)^*$: (Family of n-shocks verifying the Liu condition) There exists $r_0>0$ such for all $u\in B_{r_0}(v_R)$, there is an n-shock curve (issuing from $u$) $S_u^n\colon[0,s_u)\to \mathcal{V}$ (possibly $s_u=\infty$) parameterized by arc length. Moreover, $S_u^n(0)=u$ and the Rankine-Hugoniot jump condition holds:
\begin{align}
f(S_u^n(s))-f(u)=\sigma^n_u(s)(S_u^n(s)-u),
\end{align}
where $\sigma^n_u(s)$ is the velocity function. The map $u\mapsto s_u$ is Lipschitz on $\mathcal{V}$. Further, the maps $(s,u)\mapsto S_u^n(s)$ and $(s,u)\mapsto \sigma^n_u(s)$ are both $C^1$ on $\{(s,u)|s\in [0,s_u), u\in\mathcal{V}\}$, and the following conditions are satisfied: 
\begin{align*}
&\mbox{(a) (Liu entropy condition) } \frac{\mbox{d}}{\mbox{d}s} \sigma^n_u(s) >0,\hspace{.2in} \sigma^n_u(0)=\lambda_n(u), 
\\&\mbox{(b) (shock ``strengthens'' with $s$) }  \frac{\mbox{d}}{\mbox{d}s}\eta(u|S_u^n(s))>0, \hspace{.2in}\mbox{for all } s>0,
\\&\mbox{(c) (the shock curve cannot wrap tightly around itself)}
\\&\hspace{.5in}\mbox{For all $R>0$, there exists $\tilde{S}>0$ such that}  
\\
&\hspace{.8in}\Big\{S^n_{u}(s) \Big| s\in[0.s_u), \abs{u}\leq R \mbox{ and } \abs{S^n_u(s)}\leq R\Big\} \subseteq \Big\{S^n_u(s) \Big| \abs{u} \leq R \mbox{ and } s\leq \tilde{S}\Big\}.
\end{align*}

\item
$(\mathcal{H}2)^*$: If $(u_R,u_L)$ is an entropic Rankine-Hugoniot discontinuity with shock speed $\sigma$, then $\sigma< \lambda_n(u_L)$.

\item
$(\mathcal{H}3)^*$: If $(u_R,u_L)$ (with $u_R\in B_{r_0}(v_R)$), is an entropic Rankine-Hugoniot discontinuity with shock speed $\sigma$ verifying
\begin{align}
\sigma\geq\lambda_n(u_R),
\end{align}
then $u_L$ is in the image of $S_{u_R}^n$. In other words, there exists $s_{u_L}\in[0,s_{u_R})$ such that $S_{u_R}^n(s_{u_L})=u_L$ (and by implication, $\sigma=\sigma^n_{u_R}(s_{u_L})$).
\end{itemize}

\begin{remark}
\hfill \break
For useful remarks on these hypotheses, see \cite{move_entire_solution_system,MR3519973,Leger2011}. We include the remarks here for completeness.
\begin{itemize}
\item
Note that the system \eqref{system} verifies the hypotheses $(\mathcal{H}1)$-$(\mathcal{H}3)$ on the 1-shock family if and only if the system
\begin{align}
\begin{cases}
\partial_t u - \partial_x f(u)=0,\mbox{ } t>0,\\
u(x,0)=u^0(x) \mbox{ for } x\in\mathbb{R}.
\end{cases}
\end{align}
verifies the properties $(\mathcal{H}1)^*$-$(\mathcal{H}3)^*$ for the n-shock family. It is in this way that $(\mathcal{H}1)$-$(\mathcal{H}3)$ are dual to $(\mathcal{H}1)^*$-$(\mathcal{H}3)^*$.
\item
On top of the Liu entropy condition (Property (a) in $(\mathcal{H}1)$), we also assume Property (b), which says that the 1-shock strength grows along the 1-shock curve $S^1_{u_L}$ when measured via the pseudo-distance of the relative entropy (recall that the map $(u,v)\mapsto\eta(u|v)$ measures  $L^2$-distance somehow -- see \Cref{entropy_relative_L2_control_system}). This growth condition arises naturally in the study of admissibility criteria for systems of conservation laws. In particular, Property (b) ensures that Liu admissible shocks are entropic for the entropy $\eta$ even for moderate-to-strong shocks (see \cite{MR1600904,MR0093653,MR2053765}). 

In \cite{MR3338447}, Barker, Freist\"{u}hler, and Zumbrun show that stability  and in particular contraction fails to hold for the full Euler system if we replace Property (b) with
\begin{align}
\frac{\mbox{d}}{\mbox{d}s}\eta(S_u^1(s))>0,\hspace{.2in} s>0.
\end{align}
This shows that it is better to measure shock strength using the relative entropy rather than the entropy itself.
\item
Recall the famous Lax E-condition for an  i-shock $(u_L,u_R,\sigma)$,
\begin{align}
\lambda_i(u_R)\leq\sigma\leq\lambda_i(u_L).
\end{align}
The hypothesis $(\mathcal{H}2)$ is implied by the first half of the Lax E-condition along with the hyperbolicity of the system \eqref{system}. In addition, we do not allow for right 1-contact discontinuities. 
\item
The hypothesis $(\mathcal{H}3)$ is a statement about the well-separation of the 1-shocks from all other Rankine-Hugoniot discontinuities entropic for $\eta$; the 1-shocks do not interfere with any other shocks. In particular, $(\mathcal{H}3)$ will hold for any strictly hyperbolic system in the form \eqref{system} if all Rankine-Hugoniot discontinuities $(u_L,u_R,\sigma)$ entropic for $\eta$ lie on an i-shock curve for some $i$ and the extended Lax admissibility condition holds:
\begin{align}\label{extended_lax_admissibility_condition}
\lambda_{i-1}(u_L) \leq \sigma \leq \lambda_{i+1} (u_R),
\end{align}
where $\lambda_0\coloneqq -\infty$ and $\lambda_{n+1}\coloneqq\infty$. Moreover, we only use the first inequality in \eqref{extended_lax_admissibility_condition} and the fact that $\lambda_1(u)\leq \lambda_{i-1}(u)$ for all $u\in\mathcal{V}$ and for all $i>1$.

Furthermore, note that for \emph{any} strictly hyperbolic system in the form \eqref{system}, if  $u_R$ and $u_L$ live in a fixed compact set, then there exists $\delta>0$ such that \eqref{extended_lax_admissibility_condition} will hold if $\abs{u_R-u_L}\leq\delta$. Similarly,  for any strictly hyperbolic system endowed with a strictly convex entropy, all Rankine-Hugoniot discontinuities $(u_L,u_R,\sigma)$ entropic for $\eta$ will locally be in the form $S^i_{u_L}(s)=u_R$ for some $s>0$, and where $S^i_{u_L}$ is the i-shock curve issuing  from $u_L$. See \cite[Theorem 1.1, p.~140]{lefloch_book} and more generally \cite[p.~140-6]{lefloch_book}. For the full Euler system, $(\mathcal{H}3)$ will hold regardless of the size of the shock $(u_L,u_R)$.
\item
Note that due to the map $(s,u)\mapsto S_u^1(s)$ being Lipschitz, we have 
\begin{align}
\abs{S_u^1(s)-u}=\abs{S_u^1(s)-S_u^1(0)}\leq \mbox{Lip}\Big[(s,u)\mapsto S_u^1(s)\Big] s,
\end{align}
for all $u\in B_{r_0}(I_{-})$ and all $s\in [0,s_u)$. Equivalently,
\begin{align}\label{shock_strength_comparable_s_systems1}
\frac{1}{\mbox{Lip}\Big[(s,u)\mapsto S_u^1(s)\Big] }\abs{S_u^1(s)-u}\leq s.
\end{align}
\item
On the state space $\mathcal{V}$ where the strictly convex entropy $\eta$ is defined, the system \eqref{system} is hyperbolic. Further, by virtue of $f\in C^2(\mathcal{V})$, the eigenvalues of $\nabla f (u)$ vary continuously on the state space $\mathcal{V}$. Further, if the eigenvalue $\lambda_1(u)$ ($\lambda_n(u)$) is simple for $u\in\mathcal{V}$ (such as when the system \eqref{system} is strictly hyperbolic), the map $u\mapsto \lambda_1(u)$ ($u\mapsto \lambda_n(u)$) will be in $C^1(\mathcal{V})$ due to the implicit function theorem.
\end{itemize}
\end{remark}

We study solutions $u$ to \eqref{system} among the class of functions verifying a strong trace property (first introduced in \cite{Leger2011}):

\begin{definition}\label{strong_trace_definition}
Fix $T>0$. Let $u\colon\mathbb{R}\times[0,T)\to\mathbb{R}^n$ verify $u\in L^\infty(\mathbb{R}\times[0,T))$. We say $u$ has the \emph{strong trace property} if for every fixed Lipschitz continuous map $h\colon [0,T)\to\mathbb{R}$, there exists $u_+,u_-\colon[0,T)\to\mathbb{R}^n$ such that
\begin{align}
\lim_{n\to\infty}\int\limits_0^{t_0}\esssup_{y\in(0,\frac{1}{n})}\abs{u(h(t)+y,t)-u_+(t)}\,dt=\lim_{n\to\infty}\int\limits_0^{t_0}\esssup_{y\in(-\frac{1}{n},0)}\abs{u(h(t)+y,t)-u_-(t)}\,dt=0
\end{align}
for all $t_0\in(0,T)$.
\end{definition}

Note that for example a function $u\in L^\infty(\mathbb{R}\times[0,T))$ will satisfy the strong trace property if for each fixed $h$, the right and left limits
\begin{align}
\lim_{y\to0^{+}}u(h(t)+y,t)\hspace{.7in}\mbox{and}\hspace{.7in}\lim_{y\to0^{-}}u(h(t)+y,t)
\end{align}
exist for almost every $t$. In particular, a function $u\in L^\infty(\mathbb{R}\times[0,T))$ will have strong traces according to \Cref{strong_trace_definition} if $u$ has a representative which is in $BV_{\text{loc}}$. However, the strong trace property is weaker than $BV_{\text{loc}}$.

\section{Overview of the proofs of \Cref{RP_theorem_1} and \Cref{RP_theorem_3}}\label{section_overview}

Within the context of the relative entropy method, the theory of shifts often works by moving shocks with an artificial velocity, as opposed to the velocity dictated by the Rankine-Hugoniot jump condition. One difficulty in applying the theory of shifts to solving the Riemann problem is, what to do if the graph of a $x=h(t)$ shift function (in the x-t plane) for a particular shock intersects one of the rarefaction fans? At this point, it is not guaranteed that the states to the left and  right of the shift function are an entropic discontinuity (they might not even satisfy Rankine-Hugoniot) -- and this prevents analysis. But this is again solved using generalized-characteristic-based shifts. For example, the generalized-characteristic-based shifts for a 1-shock in $\bar{v}$ will either travel at characteristic-like speed of $u$, or they will travel to the left very quickly (super-characteristic speed). 

When the generalized-characteristic-based shift (for a 1-shock) is traveling to the left very fast, we do not have to worry about it intersecting with a rarefaction fan, which will spread out with characteristic speed. When the generalized-characteristic-based shift is traveling with characteristic speed, then we must control the speed of generalized characteristic of $u$ versus the speed the rarefaction fans in $\bar{v}$ are spreading out. Heuristically, the function $u$ goes into the first slot $\eta(\cdot|)$ of the relative entropy, and $\bar{v}$ goes into the second $\eta(|\cdot)$, and there is little connection between the two slots of the relative entropy. However, through the strong form of Lax's E-condition \eqref{lax_strong}, we can connect these two worlds of the first and second slot of the relative entropy and show that the generalized characteristic of $u$ will not intersect the rarefaction fans in the x-t plane. In fact, the analysis will depend on the quantity  $(\lambda_{i+1}(\bar{v}_{i+1})-\lambda_i(\bar{v}_i))$ if $\bar{v}$ has an i-shock $(\bar{v}_{i},\bar{v}_{i+1})$. For example, for hyperbolic systems of conservation laws where the characteristics speeds are completely separated in value, any shock will satisfy \eqref{lax_strong}. Furthermore, for such systems it is clear that a shift function traveling at the speed of a generalized characteristic for one wave family cannot intersect the rarefaction fan of a different wave family. See \Cref{RP_theorem_1}.

If $\bar{v}$ does not contain any rarefactions, then we do not have to compare the shifts to the rarefactions to make sure they are not interacting. Instead, we only need to prevent the shifts corresponding to a 1-shock from interacting with the shifts corresponding to an n-shock. We want the two shifts to stay away from each other, because if they touch and stick together then the left and right hand states to the left and right of the (now single) shift will in general not make an entropic shock. Without rarefactions in between these two shifts to separate them, we cannot use the arguments from \Cref{RP_theorem_1}. We instead study the two shifts directly. See \Cref{RP_theorem_3}.


\section{Technical Lemmas}\label{technical_lemmas}

For use throughout this paper, we define the relative flux
\begin{align}\label{Z_def}
f(a|b)\coloneqq f(a)-f(b)-\nabla f (b)(a-b),
\end{align}
for $a,b\in\mathbb{M}^{n\times 1}$. Further, for $a,b\in\mathbb{M}^{n\times 1}$, we define the relative $\nabla\eta$:
\begin{align}
\nabla\eta(a|b)\coloneqq \nabla\eta(a)-\nabla\eta(b)-[a-b]^T\nabla^2\eta(b).
\end{align}

\begin{lemma}\label{a_cond_lemma_itself}
Fix $B>0$. Then there exists a constant $C>0$ depending on $B$ such that the following holds:
 
If $u_L,u_R\in\mathcal{V}$ with $\abs{u_L},\abs{u_R}\leq B$, then whenever $\alpha,\theta\in(0,1)$ verify 
\begin{align}
\label{cond_a}
\alpha<\frac{\theta^2}{C},
\end{align}
then $R_a\coloneqq\{u | \eta(u|u_L)\leq a\eta(u|u_R)\}\subset B_{\theta}(u_L)$ for all $0<a<\alpha$. 
\end{lemma}
\begin{remark}
The set $R_a$ is compact.
\end{remark}
The proof of \Cref{a_cond_lemma_itself} is found in the proof of Lemma 4.3 in \cite{MR3519973}. 

The following Lemma gives us the entropy dissipation caused by changing the  domain of integration and translating the piecewise-smooth solution $\bar{u}$ in $x$ (by a function $X(t)$).
\begin{lemma}[Local entropy dissipation rate]\label{local_entropy_dissipation_rate_systems_RP}

Let $u,\bar{u}\in L^\infty(\mathbb{R}\times[0,T))$ be weak solutions to \eqref{system}. We assume that $u,
\bar{u}$ are entropic for the entropy $\eta$. Assume that $\bar{u}$ is Lipschitz continuous on $\{(x,t)\in\mathbb{R}\times[0,T) | x<s(t)\}$ and on $\{(x,t)\in\mathbb{R}\times[0,T) | x>s(t)\}$, where $s:[0,T)\to\mathbb{R}$ is a Lipschitz function . Assume also that $u$ verifies the strong trace property (\Cref{strong_trace_definition}).  Let $T,t_0,t_1\in\mathbb{R}$ verify $0\leq t_0< t_1 <T$. Let $h_1,h_2, X:[0,T)\to\mathbb{R}$ be Lipschitz continuous functions with the property that $h_2(t)-h_1(t)>0$ for all $t\in(t_0,t_1)$. We also require that if $t_0\neq0$,  then $h_1(t_0)=h_2(t_0)$. Further assume that for all $t\in[t_0,t_1]$, $s(t)-X(t)$ is not in the open set $(h_1(t),h_2(t))$.

Then,
\begin{equation}
\begin{aligned}\label{local_compatible_dissipation_calc_RP}
\int\limits_{t_0}^{t_1} \bigg[q(u(h_1(t)+,t);\bar{u}((h_1(t)+X(t))+,t))-q(u(h_2(t)-,t);\bar{u}((h_2(t)+X(t))-,t))+
\\
\dot{h}_2(t)\eta(u(h_2(t)-,t)|\bar{u}((h_2(t)+X(t))-,t))-
\\
\dot{h}_1(t)\eta(u(h_1(t)+,t)|\bar{u}((h_1(t)+X(t))+,t))\bigg]\,dt
\\
\geq
\int\limits_{h_1(t_1)}^{h_2(t_1)}\eta(u(x,t_1)|\bar{u}(x+X(t_1),t_1))\,dx
\\
-\int\limits_{h_1(t_0)}^{h_2(t_0)}\eta(u^0(x)|\bar{u}^0(x))\,dx
+
\\
\int\limits_{t_0}^{t_1}\int\limits_{h_1(t)}^{h_2(t)}\Bigg(\partial_x \bigg|_{(x+X(t),t)}\hspace{-.45in} \nabla\eta(\bar{u}(x,t))\Bigg)f(u(x,t)|\bar{u}(x+X(t),t))
\\
+\Bigg(2\partial_x\bigg|_{(x+X(t),t)}\hspace{-.45in}\bar{u}^T(x,t)\dot{X}(t)\Bigg)\nabla^2\eta(\bar{u}(x+X(t),t))[u(x,t)-\bar{u}(x+X(t),t)]\,dxdt.
\end{aligned}
\end{equation}

\end{lemma}
\begin{remark}
If $t_0\neq 0$, then $h_1(t_0)=h_2(t_0)$ and 
\begin{align}
\int\limits_{h_1(t_0)}^{h_2(t_0)}\eta(u^0(x)|\bar{u}^0(x))\,dx=0.
\end{align}
\end{remark}
\begin{remark}
\Cref{local_entropy_dissipation_rate_systems_RP}, and in particular \eqref{local_compatible_dissipation_calc_RP}, are not true if $h_1(t)=h_2(t)$ for all $t$ in some open interval.

To see this, consider the following simple example: Let $\bar{u}\coloneqq v$ for some constant state $v\in\mathbb{R}^n$. Let $(u_L,u_R,\sigma(u_L,u_R))$ be a shock entropic for the entropy $\eta$. Define
\begin{align}
 u(x,t)\coloneqq
  \begin{cases}
   u_L & \text{if } x<\sigma(u_L,u_R) t\\
    u_R & \text{if } x>\sigma(u_L,u_R) t.
  \end{cases}
\end{align}

Choose $h_1(t)\coloneqq h_2(t)\coloneqq\sigma(u_L,u_R) t$.

With these choices, the right hand side of \eqref{local_compatible_dissipation_calc_RP} vanishes. 

The left hand side of \eqref{local_compatible_dissipation_calc_RP} becomes 
\begin{align}\label{leftHS_becomes_RP}
\int\limits_{t_0}^{t_1} \bigg[q(u_R;v)-q(u_L;v)-\sigma(u_L,u_R)(\eta(u_R|v)-\eta(u_L|v))\bigg]\,dt.
\end{align}

Note that because $u$ is entropic for the entropy $\eta$, $u$ is also entropic for the entropy
\begin{align}\label{relative_entropy_with_v_RP}
u\mapsto \eta(u|v).
\end{align}
This follows because the map \eqref{relative_entropy_with_v_RP} is simply the function $\eta(u)$ plus a term (affine) linear in $u$.

Thus, the shock $(u_L,u_R,\sigma(u_L,u_R))$ is entropic for \eqref{relative_entropy_with_v_RP}. This implies that  
\begin{align}\label{shock_entropic_rel_entrop_RP_example}
q(u_R;v)-q(u_L;v)-\sigma(u_L,u_R)(\eta(u_R|v)-\eta(u_L|v))\leq0.
\end{align}

By choosing a shock $(u_L,u_R,\sigma(u_L,u_R))$ such that \eqref{shock_entropic_rel_entrop_RP_example} is \emph{strictly} negative, we have shown that \eqref{local_compatible_dissipation_calc_RP} does not hold (recall \eqref{leftHS_becomes_RP}). 

Intuitively, why does \eqref{local_compatible_dissipation_calc_RP} fails to hold when $h_1(t)=h_2(t)$? This is because when $h_1(t)\neq h_2(t)$ the function $h_1$ thinks that if it moves to the right (or left), it is reducing (or creating more of) the entropy in the integral 
\begin{align}\label{integral_think_of_mass_intuition}
\int\limits_{h_1(t)}^{h_2(t)}\eta(u(x,t)|\bar{u}(x+X(t),t))\,dx,
\end{align} 
by contracting (or expanding) the domain of integration. And similarly for $h_2$. However, if for a positive amount of time $h_1(t)=h_2(t)$, then \eqref{integral_think_of_mass_intuition} is always zero and no mass is created or destroyed.

However, as long as $h_1(t)=h_2(t)$ only for brief moments, a version of  \Cref{local_entropy_dissipation_rate_systems_RP} still holds. See \Cref{corollary_dissipation_RP}.

\end{remark}
\begin{proof}[Proof of \Cref{local_entropy_dissipation_rate_systems_RP}]

This proof is based on a similar argument in \cite{scalar_move_entire_solution}.

\hfill\break
\uline{Step 1}
\hfill\break

We first show that for all positive, Lipschitz continuous test functions $\phi:\mathbb{R}\times[0,T)\to\mathbb{R}$ with compact support and that vanish on the set $\{(x,t)\in\mathbb{R}\times[0,T) | x=s(t)-X(t)\}$, we have
\begin{equation}
\begin{aligned}\label{combined1}
&\int\limits_{0}^{T} \int\limits_{-\infty}^{\infty} [\partial_t \phi \eta(u(x,t)|\bar{u}(x+X(t),t))+\partial_x \phi q(u(x,t);\bar{u}(x+X(t),t))]\,dxdt 
\\&\hspace{1in}+  \int\limits_{-\infty}^{\infty}\phi(x,0)\eta(u^0(x)|\bar{u}^0(x))\,dx \\
&\hspace{1in}\geq
\int\limits_{0}^{T} \int\limits_{-\infty}^{\infty}\phi\Bigg[\Bigg(
\partial_x \bigg|_{(x+X(t),t)} \hspace{-.45in}\nabla\eta(\bar{u}(x,t))\Bigg) f(u(x,t)|\bar{u}(x+X(t),t))
\\
&\hspace{1in}+\Bigg(2\partial_x\bigg|_{(x+X(t),t)}\hspace{-.45in}\bar{u}^T(x,t)\dot{X}(t)\Bigg)\nabla^2\eta(\bar{u}(x+X(t),t))[u(x,t)-\bar{u}(x+X(t),t)]
\Bigg]\,dxdt.
\end{aligned}
\end{equation}
Note that \eqref{combined1} is the analogue in our case of the key estimate used in Dafermos's proof of weak-strong stability, which gives a relative version of the entropy inequality (see equation (5.2.10) in \cite[p.~122-5]{dafermos_big_book}). The proof of \eqref{combined1} is based on the famous weak-strong stability proof of Dafermos and DiPerna \cite[p.~122-5]{dafermos_big_book}. We then modify the Dafermos and DiPerna proof as in \cite{scalar_move_entire_solution} to allow for the translation of the solution $\bar{u}$ by the function $X$ and to account for the additional entropy this creates.

%

Note that on the complement of the set $\{(x,t)\in\mathbb{R}\times[0,T) | x=s(t)\}$, $\bar{u}$ is smooth  and so we have the exact equalities,
\begin{align}
\partial_t\bigg|_{(x,t)}\hspace{-.21in}\big(\bar{u}(x,t)\big)+\partial_x\bigg|_{(x,t)}\hspace{-.21in}\big(f(\bar{u}(x,t))\big)&=0,\label{solves_equation}\\
\partial_t\bigg|_{(x,t)}\hspace{-.21in}\big(\eta(\bar{u}(x,t))\big)+\partial_x\bigg|_{(x,t)}\hspace{-.21in}\big(q(\bar{u}(x,t))\big)&=0.\label{solves_entropy}
\end{align}

Thus for any Lipschitz continuous function $X: [0,T)\to\mathbb{R}$ with $X(0)=0$  we have on the complement of the set $\{(x,t)\in\mathbb{R}\times[0,T) | x=s(t)-X(t)\}$, 
\begin{equation}
\begin{aligned}\label{solves_equation_shift}
\partial_t\bigg|_{(x,t)}\hspace{-.21in}&\big(\bar{u}(x+X(t),t)\big)+\partial_x\bigg|_{(x,t)}\hspace{-.21in}\big(f(\bar{u}(x+X(t),t))\big)=
\\
&\hspace{1.5in}\Bigg(\partial_x\bigg|_{(x+X(t),t)}\hspace{-.45in}\big(\bar{u}(x,t)\big)\Bigg)\dot{X}(t),
\end{aligned}
\end{equation}
and
\begin{equation}
\begin{aligned}\label{solves_entropy_shift}
\partial_t\bigg|_{(x,t)}\hspace{-.21in}&\big(\eta(\bar{u}(x+X(t),t))\big)+\partial_x\bigg|_{(x,t)}\hspace{-.21in}\big(q(\bar{u}(x+X(t),t))\big)=
\\
&\nabla\eta(\bar{u}(x+X(t),t))\Bigg(\partial_x\bigg|_{(x+X(t),t)}\hspace{-.45in}\big(\bar{u}(x,t)\big)\Bigg)\dot{X}(t).
\end{aligned}
\end{equation}

We can now imitate the weak-strong stability proof in \cite[p.~122-5]{dafermos_big_book}, using \eqref{solves_equation_shift} and \eqref{solves_entropy_shift} instead of \eqref{solves_equation} and \eqref{solves_entropy}. This gives \eqref{combined1}. For more details, the reader can refer to \cite{move_entire_solution_system}, where the  computation is done under the additional assumption that the system \eqref{system} has a source term $G$. Due to considering the source term $G$, the work \cite{move_entire_solution_system} assumes that the entropy $\eta \in C^3(\mathcal{V})$, but the computations go through unchanged if we take  $G\equiv 0$ and $\eta \in C^2(\mathcal{V})$.

\hfill\break
\uline{Step 2}
\hfill\break

We will now test \eqref{combined1} with some particular test functions. 
The rest of the proof of \Cref{local_entropy_dissipation_rate_systems_RP} is decomposed into two cases: 
\hfill\break
\emph{Case 1} $t_0=0$, $h_1(t_0)<h_2(t_0)$ 
\hfill\break
and 
\hfill\break
\emph{Case 2} $h_1(t_0)=h_2(t_0)$ 
\hfill\break
\hfill\break
We start with \emph{Case 1}.
\hfill \break
\emph{Case 1} $t_0=0$, $h_1(t_0)<h_2(t_0)$
\hfill \break

Choose $t^*\in(t_0,t_1)$.

Define 
\begin{align}
\delta\coloneqq \inf_{t\in\big[t_0,t^*+\frac{t_1-t^*}{2}\big]}(h_2(t)-h_1(t)).
\end{align}
Note $\delta>0$.

Choose $0<\epsilon<\min\{\frac{1}{2}\delta,\frac{t_1-t^*}{2}\}$.

We apply the test function $\omega^0(t)\chi(x,t)$ to \eqref{combined1}, where

\begin{align}
 \omega^0(t)\coloneqq
  \begin{cases}
   1 & \text{if } 0\leq t< t^*\\
   \frac{1}{\epsilon}(t^*-t)+1 & \text{if } t^*\leq t < t^*+\epsilon\\
   0 & \text{if } t^*+\epsilon \leq t.
  \end{cases}
\end{align}
and
\begin{align}\label{chi_RP}
 \chi(x,t)\coloneqq
  \begin{cases}
   0 & \text{if } x<h_1(t)\\
   \frac{1}{\epsilon}(x-h_1(t)) & \text{if } h_1(t)\leq x < h_1(t)+\epsilon\\
   1 & \text{if } h_1(t)+\epsilon\leq x \leq h_2(t) -\epsilon\\
   -\frac{1}{\epsilon}(x-h_2(t)) & \text{if } h_2(t)-\epsilon<x\leq h_2(t)\\
   0 & \text{if } h_2(t)<x.
  \end{cases}
\end{align}

The function $\omega^0$ is modeled from \cite[p.~124]{dafermos_big_book}. The function $\chi$ is from \cite[p.~765]{Leger2011_original}. 

We receive,
\begin{equation}
\begin{aligned}\label{local_plugged_test_new_case}
&\int\limits_{0}^{t^*} \Bigg[-\int\limits_{h_1(t)}^{h_1(t)+\epsilon}\frac{1}{\epsilon}\dot{h}_1(t)\eta(u(x,t)|\bar{u}(x+X(t),t))\,dx
+\int\limits_{h_1(t)}^{h_1(t)+\epsilon}\frac{1}{\epsilon}q(u(x,t);\bar{u}(x+X(t),t))\,dx
\\
&\hspace{.5in}+\int\limits_{h_2(t)-\epsilon}^{h_2(t)}\frac{1}{\epsilon}\dot{h}_2(t)\eta(u(x,t)|\bar{u}(x+X(t),t))\,dx-\int\limits_{h_2(t)-\epsilon}^{h_2(t)}\frac{1}{\epsilon}q(u(x,t);\bar{u}(x+X(t),t))\,dx\Bigg]\,dt
\\
&\hspace{.5in}+
\int\limits_{h_1(0)}^{h_2(0)}\eta(u^0(x)|\bar{u}^0(x))\,dx
-
\int\limits_{t^*}^{t^*+\epsilon}\frac{1}{\epsilon}\int\limits_{h_1(t)}^{h_2(t)}\eta(u(x,t)|\bar{u}(x+X(t),t))\,dxdt
+\mathcal{O}(\epsilon)
\\
&\hspace{2in}\geq
\int\limits_{0}^{t^*}\int\limits_{h_1(t)}^{h_2(t)}\mbox{RHS}\,dxdt,
\end{aligned}
\end{equation}
where RHS represents everything being multiplied by $\phi$ in the integral on the right hand side of \eqref{combined1}. 

Recall the convexity of $\eta$. Furthermore, remark that for weak solutions $u$ to \eqref{system}, the map $t\mapsto u(\cdot,t)$ is continuous in $L^\infty$ weak-* . Thus, from these two facts we have the following lower-semicontinuity property for $r\in[0,T)$:
\begin{align}\label{semicontinuity_eta_RP} 
\int\limits_{h_1(r)}^{h_2(r)}\eta(u(x,r)|\bar{u}(x+X(r),r))\,dx \leq \liminf_{s\to r} \int\limits_{h_1(s)}^{h_2(s)}\eta(u(x,s)|\bar{u}(x+X(s),s))\,dx.
\end{align}

Let $\epsilon\to0$ in \eqref{local_plugged_test_new_case}.

We use the dominated convergence, the Lebegue differentiation theorem, and recall that $u$ satisfies the strong trace property (\Cref{strong_trace_definition}). This yields,

\begin{equation}
\begin{aligned}\label{local_plugged_test_omega_0_3_dropped}
&\int\limits_{t_0}^{t^*} \bigg[q(u(h_1(t)+,t);\bar{u}((h_1(t)+X(t))+,t))-q(u(h_2(t)-,t);\bar{u}((h_2(t)+X(t))-,t))
\\
&\hspace{.4in}+\dot{h}_2(t)\eta(u(h_2(t)-,t)|\bar{u}((h_2(t)+X(t))-,t))\\
&\hspace{.4in}-\dot{h}_1(t)\eta(u(h_1(t)+,t)|\bar{u}((h_1(t)+X(t))+,t))\bigg]\,dt
+\int\limits_{h_1(0)}^{h_2(0)}\eta(u^0(x)|\bar{u}^0(x))\,dx
\\
&\hspace{2in}\geq
\int\limits_{h_1(t^*)}^{h_2(t^*)}\eta(u(x,t^*)|\bar{u}(x+X(t^*),t^*))\,dx
+
\int\limits_{t_0}^{t^*}\int\limits_{h_1(t)}^{h_2(t)}\mbox{RHS}\,dxdt,
\end{aligned}
\end{equation}
where we used \eqref{semicontinuity_eta_RP} to take the limit of the term 
\begin{align}
\int\limits_{t^{*}}^{t^{*}+\epsilon}\frac{1}{\epsilon}\int\limits_{h_1(t)}^{h_2(t)}\eta(u(x,t)|\bar{u}(x+X(t),t))\,dxdt
\end{align}
for every $t^{*}$ and not just almost every $t^{*}$.

We let $t^{*}\to t_1$ in \eqref{local_plugged_test_omega_0_3_dropped}. Recall the dominated convergence theorem, and again use \eqref{semicontinuity_eta_RP} to handle the term
\begin{align}
\int\limits_{h_1(t^*)}^{h_2(t^*)}\eta(u(x,t^*)|\bar{u}(x+X(t^*),t^*))\,dx
\end{align}

This yields \eqref{local_compatible_dissipation_calc_RP}.

\hfill \break
\emph{Case 2} $h_1(t_0)=h_2(t_0)$
\hfill \break

Choose $t^*,t^{**}\in(t_0,t_1)$ with $t^{**}<t^*$.

Define 
\begin{align}
\delta\coloneqq \inf_{t\in\big[t^{**},t^*+\frac{t_1-t^*}{2}\big]}(h_2(t)-h_1(t)).
\end{align}
Note $\delta>0$. 

Choose $0<\epsilon<\min\{\frac{1}{2}\delta,\frac{t_1-t^*}{2}\}$.

We repeat the above calculations, but instead of using $\omega^0$, we use $\omega$:
\begin{align}
 \omega(t)\coloneqq
  \begin{cases}
     0 & \text{if } 0\leq t < t^{**}\\
   \frac{1}{\epsilon}(t-t^{**}) & \text{if } t^{**}\leq t < t^{**}+\epsilon\\
   1 & \text{if }t^{**}+\epsilon \leq t< t_1\\
   \frac{1}{\epsilon}(t_1-t)+1 & \text{if } t_1\leq t < t_1+\epsilon\\
   0 & \text{if } t_1+\epsilon \leq t.
  \end{cases}
\end{align}


The function $\omega$ is from \cite[p.~765]{Leger2011_original}.

The function $\chi$ \eqref{chi_RP} is used exactly as it is.

We test \eqref{combined1} with the test function $\omega(t)\chi(x,t)$. This gives us,
\begin{equation}
\begin{aligned}\label{local_plugged_test_omega_0_case4}
\int\limits_{t^{**}}^{t^*} \Bigg[-\int\limits_{h_1(t)}^{h_1(t)+\epsilon}\frac{1}{\epsilon}\dot{h}_1(t)\eta(u(x,t)|\bar{u}(x+X(t),t))\,dx+
\int\limits_{h_1(t)}^{h_1(t)+\epsilon}\frac{1}{\epsilon}q(u(x,t);\bar{u}(x+X(t),t))\,dx
\\
+\int\limits_{h_2(t)-\epsilon}^{h_2(t)}\frac{1}{\epsilon}\dot{h}_2(t)\eta(u(x,t)|\bar{u}(x+X(t),t))\,dx-\int\limits_{h_2(t)-\epsilon}^{h_2(t)}\frac{1}{\epsilon}q(u(x,t);\bar{u}(x+X(t),t))\,dx\Bigg]\,dt
\\
+
\int\limits_{t^{**}}^{t^{**}+\epsilon}\frac{1}{\epsilon}\int\limits_{h_1(t)}^{h_2(t)}\eta(u(x,t)|\bar{u}(x+X(t),t))\,dxdt
-
\int\limits_{t^*}^{t^*+\epsilon}\frac{1}{\epsilon}\int\limits_{h_1(t)}^{h_2(t)}\eta(u(x,t)|\bar{u}(x+X(t),t))\,dxdt
+\mathcal{O}(\epsilon)
\\
\geq
\int\limits_{t^{**}}^{t^*}\int\limits_{h_1(t)}^{h_2(t)}\mbox{RHS}\,dxdt.
\end{aligned}
\end{equation}

Note that we can estimate
\begin{align}\label{estimate_bdy_term_RP_case4}
\abs{\int\limits_{t^{**}}^{t^{**}+\epsilon}\frac{1}{\epsilon}\int\limits_{h_1(t)}^{h_2(t)}\eta(u(x,t)|\bar{u}(x+X(t),t))\,dxdt}
\leq
C\sup_{t\in[t^{**},t^{**}+\epsilon]}\abs{h_2(t)-h_1(t)},
\end{align}
for some constant $C>0$. 

We combine \eqref{local_plugged_test_omega_0_case4}, \eqref{estimate_bdy_term_RP_case4} to get,

\begin{equation}
\begin{aligned}\label{local_plugged_test_omega_0_2_case4}
\int\limits_{t^{**}}^{t^*} \Bigg[-\int\limits_{h_1(t)}^{h_1(t)+\epsilon}\frac{1}{\epsilon}\dot{h}_1(t)\eta(u(x,t)|\bar{u}(x+X(t),t))\,dx+
\int\limits_{h_1(t)}^{h_1(t)+\epsilon}\frac{1}{\epsilon}q(u(x,t);\bar{u}(x+X(t),t))\,dx
\\
+\int\limits_{h_2(t)-\epsilon}^{h_2(t)}\frac{1}{\epsilon}\dot{h}_2(t)\eta(u(x,t)|\bar{u}(x+X(t),t))\,dx-\int\limits_{h_2(t)-\epsilon}^{h_2(t)}\frac{1}{\epsilon}q(u(x,t);\bar{u}(x+X(t),t))\,dx\Bigg]\,dt
\\
+
C\sup_{t\in[t^{**},t^{**}+\epsilon]}\abs{h_2(t)-h_1(t)}
+\mathcal{O}(\epsilon)
-
\int\limits_{t^*}^{t^*+\epsilon}\frac{1}{\epsilon}\int\limits_{h_1(t)}^{h_2(t)}\eta(u(x,t)|\bar{u}(x+X(t),t))\,dxdt
\\
\geq
\int\limits_{t^{**}}^{t^*}\int\limits_{h_1(t)}^{h_2(t)}\mbox{RHS}\,dxdt.
\end{aligned}
\end{equation}

Let $\epsilon\to0$ in \eqref{local_plugged_test_omega_0_2_case4}.

We again use the dominated convergence, the Lebegue differentiation theorem, and recall that $u$ satisfies the strong trace property (\Cref{strong_trace_definition}). This yields,

\begin{equation}
\begin{aligned}\label{local_plugged_test_omega_0_3_case4}
&\int\limits_{t^{**}}^{t^*} \bigg[q(u(h_1(t)+,t);\bar{u}((h_1(t)+X(t))+,t))-q(u(h_2(t)-,t);\bar{u}((h_2(t)+X(t))-,t))
\\
&\hspace{1in}+\dot{h}_2(t)\eta(u(h_2(t)-,t)|\bar{u}((h_2(t)+X(t))-,t))
\\
&\hspace{1in}-\dot{h}_1(t)\eta(u(h_1(t)+,t)|\bar{u}((h_1(t)+X(t))+,t))\bigg]\,dt
+C\abs{h_2(t^{**})-h_1(t^{**})}\\
&\hspace{1in}\geq
\int\limits_{h_1(t^{*})}^{h_2(t^{*})}\eta(u(x,t^{*})|\bar{u}(x+X(t^{*}),t^{*}))\,dx
+
\int\limits_{t^{**}}^{t^*}\int\limits_{h_1(t)}^{h_2(t)}\mbox{RHS}\,dxdt,
\end{aligned}
\end{equation}
where we used \eqref{semicontinuity_eta_RP} to take the limit of the term
\begin{align}
\int\limits_{t^*}^{t^*+\epsilon}\frac{1}{\epsilon}\int\limits_{h_1(t)}^{h_2(t)}\eta(u(x,t)|\bar{u}(x+X(t),t))\,dxdt
\end{align}
for every $t^{*}$ and not just almost every $t^{*}$.

We let $t^{**}\to t_0^+$ in \eqref{local_plugged_test_omega_0_3_case4}, recalling the dominated convergence theorem.

This gives,
\begin{equation}
\begin{aligned}\label{local_plugged_test_omega_0_3_case4_1}
&\int\limits_{t_0}^{t^*} \bigg[q(u(h_1(t)+,t);\bar{u}((h_1(t)+X(t))+,t))-q(u(h_2(t)-,t);\bar{u}((h_2(t)+X(t))-,t))
\\
&\hspace{1in}+\dot{h}_2(t)\eta(u(h_2(t)-,t)|\bar{u}((h_2(t)+X(t))-,t))
\\
&\hspace{1in}-\dot{h}_1(t)\eta(u(h_1(t)+,t)|\bar{u}((h_1(t)+X(t))+,t))\bigg]\,dt
\\
&\hspace{1in}\geq
\int\limits_{h_1(t^{*})}^{h_2(t^{*})}\eta(u(x,t^{*})|\bar{u}(x+X(t^{*}),t^{*}))\,dx +
\int\limits_{t^{**}}^{t^*}\int\limits_{h_1(t)}^{h_2(t)}\mbox{RHS}\,dxdt.
\end{aligned}
\end{equation}

Finally, we let $t^*\to t_1^-$ in \eqref{local_plugged_test_omega_0_3_case4_1}. We recall again the dominated convergence theorem and \eqref{semicontinuity_eta_RP}.

We receive \eqref{local_compatible_dissipation_calc_RP}.

This completes the proof of \Cref{local_entropy_dissipation_rate_systems_RP}.

\end{proof}

\begin{corollary}
\label{corollary_dissipation_RP}
Let $u,\bar{u}\in L^\infty(\mathbb{R}\times[0,T))$ be weak solutions to \eqref{system}. 
Assume that $u$ and $\bar{u}$ are entropic for the entropy $\eta$. Assume that $\bar{u}$ is Lipschitz continuous on $\{(x,t)\in\mathbb{R}\times[0,T) | x<s(t)\}$ and on $\{(x,t)\in\mathbb{R}\times[0,T) | x>s(t)\}$, where $s:[0,T)\to\mathbb{R}$ is a Lipschitz function . Assume also that $u$ verifies the strong trace property (\Cref{strong_trace_definition}).  Let $T,t_1\in\mathbb{R}$ verify $0<t_1 <T$. Let $h_1,h_2, X:[0,T)\to\mathbb{R}$ be Lipschitz continuous. We require that
\begin{itemize}
\item \begin{align}\label{ordering_shifts_required}
h_1(t)\leq h_2(t)
\end{align}
 for all $t\in[0,T)$,
\item
if $h_1(t)=h_2(t)$ for some $t\in[0,T)$, and $h_1$ and $h_2$ are both differentiable at $t$, then 
\begin{align}\label{derivs_ordered_RP}
\dot{h}_1(t)< \dot{h}_2(t).
\end{align}
\end{itemize}

Assume also that for all $t\in[0,t_1]$, $s(t)-X(t)$ is not in the open set $(h_1(t),h_2(t))$.

Then,
\begin{equation}
\begin{aligned}\label{local_compatible_dissipation_calc_RP_cor}
&\int\limits_{0}^{t_1} \bigg[q(u(h_1(t)+,t);\bar{u}((h_1(t)+X(t))+,t))-q(u(h_2(t)-,t);\bar{u}((h_2(t)+X(t))-,t))
\\
&\hspace{.5in}+\dot{h}_2(t)\eta(u(h_2(t)-,t)|\bar{u}((h_2(t)+X(t))-,t))
\\
&\hspace{.5in}-\dot{h}_1(t)\eta(u(h_1(t)+,t)|\bar{u}((h_1(t)+X(t))+,t))\bigg]\,dt
\\
&\hspace{.5in}\geq
\int\limits_{h_1(t_1)}^{h_2(t_1)}\eta(u(x,t_1)|\bar{u}(x+X(t_1),t_1))\,dx
-\int\limits_{h_1(0)}^{h_2(0)}\eta(u^0(x)|\bar{u}^0(x))\,dx
\\
&\hspace{.5in}+\int\limits_{0}^{t_1}\int\limits_{h_1(t)}^{h_2(t)}\Bigg(\partial_x \bigg|_{(x+X(t),t)}\hspace{-.45in} \nabla\eta(\bar{u}(x,t))\Bigg) f(u(x,t)|\bar{u}(x+X(t),t))
\\
&\hspace{.5in}+\Bigg(2\partial_x\bigg|_{(x+X(t),t)}\hspace{-.45in}\bar{u}^T(x,t)\dot{X}(t)\Bigg)\nabla^2\eta(\bar{u}(x+X(t),t))[u(x,t)-\bar{u}(x+X(t),t)]\,dxdt.
\end{aligned}
\end{equation}
\end{corollary}
\begin{remark}
This corollary says that the dissipation rate formula \eqref{local_compatible_dissipation_calc_RP} holds if $h_1(t)=h_2(t)$ for only a small number of $t$ values (see \Cref{shifts_repeatedly_touching_fig}).
\end{remark}

\begin{figure}[tb]
      \includegraphics[width=0.5\textwidth]{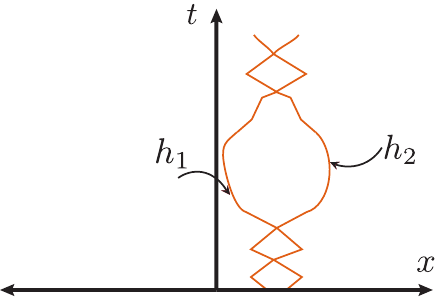}
  \caption{\Cref{corollary_dissipation_RP} allows us to consider shift functions $h_1$ and $h_2$ which occasionally touch as shown.}\label{shifts_repeatedly_touching_fig}
\end{figure}

\begin{proof}
\uline{Step 1}

Note that \eqref{ordering_shifts_required} and \eqref{derivs_ordered_RP} imply that $h_1(t)=h_2(t)$ will not occur for $t$ values where both $h_1$ and $h_2$ are differentiable. Thus the set
\begin{align}
\{\hspace{.03in}t\hspace{.03in}|h_1(t)=h_2(t)\}
\end{align}
is measure zero because Lipschitz continuous functions are differentiable almost everywhere.

\uline{Step 2}

Remark that 
\begin{align}\label{open_set_rp}
\{t\in(0,t_1)|h_1(t)\neq h_2(t)\}
\end{align}
is an open subset of $\mathbb{R}$.

Thus, we can write \eqref{open_set_rp} as a union of at most countably many disjoint open intervals:
\begin{align}\label{open_set_rp2}
\{t\in(0,t_1)|h_1(t)\neq h_2(t)\}=\bigcup_{i\in\Lambda} (x_i,y_i),
\end{align}
where $\Lambda$ is an at most countable index set, and $x_i,y_i\in\mathbb{R}$, $x_i\neq y_i$.

We now show the following two claims:

\begin{align}
\begin{cases}\label{claim_1_rp}
\mbox{If $h_1(t_1)\neq h_2(t_1)$, then there exists $i\in\Lambda$ such that}
\\
\mbox{the open interval $(x_i,y_i)$ is equal to the open interval $(x_i,t_1)$.}
\\
\mbox{Further, if $h_1(y_i)\neq h_2(y_i)$ for some $i\in\Lambda$, then $y_i=t_1$.}
\end{cases}
\end{align}

and 

\begin{align}
\begin{cases}\label{claim_2_rp}
\mbox{If $h_1(0)\neq h_2(0)$, then there exists $i\in\Lambda$ such that}
\\
\mbox{the open interval $(x_i,y_i)$ is equal to the open interval $(0,y_i)$.}
\\
\mbox{Further, if $h_1(x_i)\neq h_2(x_i)$ for some $i\in\Lambda$, then $x_i=0$.}
\end{cases}
\end{align}

The proofs of \eqref{claim_1_rp} and \eqref{claim_2_rp} are similar. We will only show \eqref{claim_1_rp}:

If $h_1(t_1)\neq h_2(t_1)$, then  by continuity of $h_1,h_2$ there exists $\alpha\in[0,t_1)$ such that $h_1(t)\neq h_2(t)$ for all $t$ in the open interval $(\alpha,t_1)$, with either $h_1(\alpha)=h_2(\alpha)$ or $\alpha=0$. By \eqref{open_set_rp2}, we must have 
\begin{align}
(\alpha,t_1)\subseteq \bigcup_{i\in\Lambda} (x_i,y_i).
\end{align}

Consider $i\in\Lambda$ such that $(x_i,y_i)\cap(\alpha,t_1)\neq \varnothing$. Then if $y_i < t_1$, we have a contradiction to the fact that the open intervals $(x_i,y_i)$ are disjoint. This proves the first part of \eqref{claim_1_rp}.

Assume now that $h_1(y_i)\neq h_2(y_i)$ for some $i\in\Lambda$. By  definition \eqref{open_set_rp2}, $y_i\leq t_1$. If $y_i < t_1$, then by continuity of $h_1,h_2$, there exists $\epsilon>0$ such that $h_1(t)\neq h_2(t)$ for all $t\in[y_i,y_i+\epsilon)$. This contradicts that the open intervals $(x_i,y_i)$ are disjoint. Recall also that $x_i\neq y_i$ for all $i$. Thus, we conclude that $y_i=t_1$.

This proves \eqref{claim_1_rp}.

\uline{Step 3}

For each $i\in\Lambda$, we apply \eqref{local_compatible_dissipation_calc_RP_1} to the time interval $(x_i,y_i)$. Note that we can do this because by \eqref{claim_2_rp}, $h_1(x_i)=h_n(x_i)$ whenever $x_i\neq0$. This gives,

\begin{equation}
\begin{aligned}\label{local_compatible_dissipation_calc_RP_1_1}
&\int\limits_{x_i}^{y_i} \bigg[q(u(h_1(t)+,t);\bar{u}((h_1(t)+X(t))+,t))-q(u(h_2(t)-,t);\bar{u}((h_2(t)+X(t))-,t))
\\
&\hspace{.2in}+\dot{h}_2(t)\eta(u(h_2(t)-,t)|\bar{u}((h_2(t)+X(t))-,t))
\\
&\hspace{.2in}-\dot{h}_1(t)\eta(u(h_1(t)+,t)|\bar{u}((h_1(t)+X(t))+,t))\bigg]\,dt
\\
&\hspace{.2in}\geq
\int\limits_{h_1(y_i)}^{h_2(y_i)}\eta(u(x,y_i)|\bar{u}(x+X(y_i),y_i))\,dx
-\int\limits_{h_1(x_i)}^{h_2(x_i)}\eta(u^0(x)|\bar{u}^0(x))\,dx
\\
&\hspace{.2in}+\int\limits_{x_i}^{y_i}\int\limits_{h_1(t)}^{h_2(t)}\Bigg(\partial_x \bigg|_{(x+X(t),t)}\hspace{-.45in} \nabla\eta(\bar{u}(x,t))\Bigg) f(u(x,t)|\bar{u}(x+X(t),t))
\\
&\hspace{.2in}+\Bigg(2\partial_x\bigg|_{(x+X(t),t)}\hspace{-.45in}\bar{u}^T(x,t)\dot{X}(t)\Bigg)\nabla^2\eta(\bar{u}(x+X(t),t))[u(x,t)-\bar{u}(x+X(t),t)]\,dxdt,
\end{aligned}
\end{equation}
where when we write the term
\begin{align}
\int\limits_{h_1(x_i)}^{h_2(x_i)}\eta(u^0(x)|\bar{u}^0(x))\,dx
\end{align}
we have again used that $h_1(x_i)=h_2(x_i)$ for $x_i\neq 0$, in which case this term vanishes.

We then sum both sides of the inequality \eqref{local_compatible_dissipation_calc_RP_1_1} over all $i\in\Lambda$. Recall that the set 
\begin{align}
\{\hspace{.03in}t\hspace{.03in}|h_1(t)=h_2(t)\}
\end{align}
has measure zero. Recall also \eqref{claim_1_rp} and \eqref{claim_2_rp}. Further, recall that the intervals $(x_i,y_i)$ are disjoint. Lastly, note that terms of the form
\begin{align}
\int\limits_{h_1(t)}^{h_2(t)}\eta(u|\bar{u})\,dx
\end{align}
equal zero when $h_1(t)=h_2(t)$.
 
This proves \eqref{local_compatible_dissipation_calc_RP_cor}.
\end{proof}

\section{Construction of the shift}\label{construction_of_the_shift}

In this section, we prove

\begin{proposition}[Existence of the shift functions]\label{systems_entropy_dissipation_room}

Fix $T>0$.   Assume $u$ is a weak solution to \eqref{system}. Assume $u$ is entropic for the entropy $\eta$, and $u$ has strong traces (\Cref{strong_trace_definition}).

Let $(u_{L,1},u_{R,1},\sigma^1(u_{L,1},u_{R,1}))$ be a 1-shock verifying the hypotheses $(\mathcal{H})$ and let 

$(u_{L,n},u_{R,n},\sigma^n(u_{L,n},u_{R,n}))$ be an n-shock verifying  the hypotheses $(\mathcal{H})^*$.

Assume also that there exists $\rho>0$ such that 
\begin{align}\label{gap_local_case}
r_i>\rho,
\end{align}
for $i=1$ and $i= n$ and where $r_i$ satisfies $S^1_{u_{L,i}}(r_i)=u_{R,i}$.

Then, there exist positive constants $a_{1,*},a_{n,*}$ such that for all $a_1\in(0,a_{1,*})$ and all $a_n\in(a_{n,*},\infty)$, there are Lipschitz continuous maps $h_1,h_n: [0,T)\to\mathbb{R}$ with $h_1(0)=h_n(0)=0$  such that for almost every $t$,
\begin{equation}
\begin{aligned}\label{dissipation_negative_claim}
a_1\big(q(u^1_+;u_{R,1})&-\dot{h}_1(t)\eta(u^1_+|u_{R,1})\big)-q(u^1_-;u_{L,1})+\dot{h}_1(t)\eta(u^1_-|u_{L,1}) \leq \\
& -c_1 \abs{\sigma^1(u_{L,1},u_{R,1})-\dot{h}_1(t)}^2,
\end{aligned}
\end{equation}
and 
\begin{equation}
\begin{aligned}\label{dissipation_negative_claim_n}
\frac{1}{a_n}\big(q(u^n_+;u_{R,n})&-\dot{h}_n(t)\eta(u^n_+|u_{R,n})\big)-q(u^n_-;u_{L,n})+\dot{h}_n(t)\eta(u^n_-|u_{L,n}) \leq \\
& -c_n \abs{\sigma^n(u_{L,n},u_{R,n})-\dot{h}_n(t)}^2,
\end{aligned}
\end{equation}
where $u_{\pm}^i\coloneqq u(u(h_i(t)\pm,t)$ for $i=1,n$. The constants $c_i>0$ depend on $\norm{u}_{L^\infty}$, $\rho$, $\abs{u_{R,i}}$, $\abs{u_{L,i}}$, and $a_i$. The constants $a_{i,*}$ depend on $\norm{u}_{L^\infty}$, $\abs{u_{L,i}},\abs{u_{R,i}}$, and $\abs{u_{R,i}-u_{L,i}}$, for $i=1,n$.

For each  $t\in[0,T)$ either $\dot{h}_1(t)<\inf\lambda_1$ or $(u^1_+,u^1_-,\dot{h}_1)$ is a 1-shock with $u^1_-\in \{u | \eta(u|u_{L,1})\leq a_1\eta(u|u_{R,1})\}$ (possibly $u^1_+=u^1_-$ and $\dot{h}_1=\lambda_1(u^1_\pm)$). Similarly, for each  $t\in[0,T)$ either $\dot{h}_n(t)>\sup\lambda_n$ or $(u^n_+,u^n_-,\dot{h}_n)$ is an n-shock with $u^n_+\in \{u | \eta(u|u_{R,n})\leq a_n\eta(u|u_{L,n})\}$ (possibly $u^n_+=u^n_-$ and $\dot{h}_n=\lambda_n(u^n_\pm)$).

Moreover,
\begin{align}\label{ordering_h_1_h_n_RP}
h_1(t)\leq h_n(t)
\end{align}
for all $t\in[0,T)$.

\end{proposition}

The proof of \Cref{systems_entropy_dissipation_room} is based on the following Lemma proved in \cite{move_entire_solution_system}.
\begin{lemma}[from \cite{move_entire_solution_system}]
\label{dissipation_negative_theorem}
Assume the hypotheses $(\mathcal{H})$ hold. 

Let $B,\rho>0$. Then there exists a constant $a_*\in(0,1)$ depending on $B$ and $\rho$ such that the following is true:

For any $a\in(0,a_*)$, there exists a constant $c_1$ depending on $B$, $\rho$, and $a$ such that
\begin{equation}
\begin{aligned}\label{dissipation_negative}
a\big(q(S^1_u(s);S^1_{u_L}(s_R))&-\sigma^1_u(s)\eta(S^1_u(s)|S^1_{u_L}(s_R))\big)-q(u;u_L)+\sigma^1_u(s)\eta(u|u_L) \leq \\
& -c_1\abs{\sigma^1_{u_L}(s_R)-\sigma^1_u(s)}^2,
\end{aligned}
\end{equation}
for all  $u_L\in\mathcal{V}$ with $\abs{u_L}\leq B$, all $u\in\{u | \eta(u|u_L)\leq a\eta(u|S^1_{u_L}(s_R))\}$, any $s\in[0,B]$, and any $s_R\in[\rho,B]$. 

Moreover,
\begin{equation}\label{dissipation_negative_boundary_of_convex_set}
a\big(q(u;S^1_{u_L}(s_R))-\lambda_1(u)\eta(u|S^1_{u_L}(s_R))\big)-q(u;u_L)+\lambda_1(u)\eta(u|u_L)\leq -c_1,
\end{equation}
for all $u\in\{u | \eta(u|u_L)\leq a\eta(u|S^1_{u_L}(s_R))\}$ and for the same constant $c_1$.

\end{lemma}

\Cref{dissipation_negative_theorem} follows from the proof of Lemma 4.3 in \cite{MR3519973}, but the proof of  \Cref{dissipation_negative_theorem} (as proved in \cite{move_entire_solution_system}) keeps careful track of the dependencies on the constants and makes sure in the calculations  to leave some extra negativity in the entropy dissipation lost at the shock $(u_L,u_R,\sigma_{L,R})$ (thus we have a negative right hand side in our \eqref{dissipation_negative} and \eqref{dissipation_negative_boundary_of_convex_set}). The idea of creating negative entropy dissipation is related to the previous works \cite{2017arXiv171207348K,scalar_move_entire_solution,move_entire_solution_system}.

The proof of \Cref{dissipation_negative_theorem} is powered by \Cref{entropy_lost_right_side_1_shock}:

\begin{lemma}[from \cite{move_entire_solution_system}]\label{entropy_lost_right_side_1_shock}
Assume the system \eqref{system} satisfies the hypothesis $(\mathcal{H}1)$. Fix $B,\rho>0$. Then there exists $k,\delta_0>0$ depending on $B$ and $\rho$ such that for any $\delta\in(0,\delta_0]$, $u\in\mathcal{V}$ with $\abs{u}\leq B$ and for any $s_0\in(\rho,B)$ and $s\geq0$,
\begin{equation}
\begin{aligned}\label{entropy_lost_right_side_1_shock_inequalities}
&q(S_u^1(s);S_u^1(s_0))-\sigma_u^1(s)\eta(S_u^1(s)|S_u^1(s_0))\leq -k\abs{\sigma_u^1(s)-\sigma_u^1(s_0)}^2,\hspace{.2in}\mbox{for } \abs{s-s_0}<\delta,\\
&q(S_u^1(s);S_u^1(s_0))-\sigma_u^1(s)\eta(S_u^1(s)|S_u^1(s_0))\leq -k\delta\abs{\sigma_u^1(s)-\sigma_u^1(s_0)},\hspace{.2in}\mbox{for } \abs{s-s_0}\geq\delta.
\end{aligned}
\end{equation}
\end{lemma}
The formula \eqref{entropy_lost_right_side_1_shock_inequalities} is a modification on a key lemma due to DiPerna \cite{MR523630}. The proof of \Cref{entropy_lost_right_side_1_shock} in \cite{move_entire_solution_system} is based on the proof of a very similar result in \cite[p.~387-9]{MR3519973}. The proof in \cite{move_entire_solution_system} modifies the proof in \cite[p.~387-9]{MR3519973} -- being careful to keep the constants $k$ and $\delta_0$ uniform in $s_0$ and $u$.

\subsection{Proof of \Cref{systems_entropy_dissipation_room}}
The proof of \eqref{dissipation_negative_claim} and \eqref{dissipation_negative_claim_n} is based on the work \cite{move_entire_solution_system}. The result \eqref{ordering_h_1_h_n_RP} is a novel contribution.

\uline{Proof of \eqref{dissipation_negative_claim}}

We will use \Cref{dissipation_negative_theorem}. The 1-shock $(u_{L,1},u_{R,1},\sigma^1(u_{L,1},u_{R,1}))$ in \Cref{systems_entropy_dissipation_room} will play the role of $(u_L,S^1_{u_L}(s_R))$ in \Cref{dissipation_negative_theorem}. Take $R\coloneqq \max\{\norm{u}_{L^\infty},\abs{u_{L,1}}\}$ and then take the $\tilde{S}$ corresponding to this $R$ as in Property (c) of $(\mathcal{H}1)$. Define the $B$ in \Cref{dissipation_negative_theorem} to be $B\coloneqq \max\{R,\tilde{S},\abs{u_{R,1}}\}$. Then, we have that for all $(u_-,u_+,\sigma)$ 1-shock with $u_-,u_+ < R$, there exists $s\in(0,B)$ such that $u_+=S^1_{u_-}(s)$. Further, note that $B$ depends on $\norm{u}_{L^\infty}$ and $\abs{u_{L,1}}$. 

Then,  we will have a constant $0<a_{1,*}<1$ as in \Cref{dissipation_negative_theorem}. Here, $a_{1,*}$ is playing the role of the $a_*$ in \Cref{dissipation_negative_theorem}.  Then, as in the statement of \Cref{systems_entropy_dissipation_room}, we choose any $a_1\in(0,a_{1,*})$.

Throughout this proof, $c$ denotes a generic constant that depends on $\norm{u}_{L^\infty}$, $\rho$, $\abs{u_{R,1}}$, $\abs{u_{L,1}}$, and $a_1$.

\uline{Step 1}

We now show that for any $\gamma_0>0$,

\begin{align}\label{bound_on_inf}
\inf \eta(u|u_L)- a_1\eta(u|u_R) \geq c_4\gamma_0^2
\end{align}
for a constant $c_4>0$, where the infimum runs over all $(u,u_L,u_R)$ such that $\mbox{dist}(u,\{w|\eta(w|u_L)\leq a_1\eta(w|u_R)\})\geq \gamma_0$ and $\abs{u_L},\abs{u_R}\leq B$. Here, $B$ is from \Cref{dissipation_negative_theorem} and the distance $\mbox{dist}(x,A)$ between a point $x$ and a set $A$ is defined in the usual way,
\begin{align}
\mbox{dist}(x,A) \coloneqq \inf_{y\in A} \abs{x-y}.
\end{align}

Consider any triple $(u,u_L,u_R)$ such that  $\mbox{dist}(u,\{w|\eta(w|u_L)\leq a_1\eta(w|u_R)\})\geq \gamma_0$ and $\abs{u_L},\abs{u_R}\leq B$. 

By \Cref{dissipation_negative_theorem}, the set $\{w|\eta(w|u_L)\leq a_1\eta(w|u_R)\}$ is compact. Thus, there exists $w_0\in\{w|\eta(w|u_L)\leq a_1\eta(w|u_R)\}$ such that
\begin{align}
\abs{u-w_0}= \mbox{dist}(u,\{w|\eta(w|u_L)\leq a_1\eta(w|u_R)\}).
\end{align}

We Taylor expand the function 
\begin{align}
\Gamma(u)\coloneqq \eta(u|u_L)- a_1\eta(u|u_R)
\end{align}
around the point $w_0$:

\begin{align}
\Gamma(u)=\Gamma(w_0)+\nabla\Gamma(w_0)(u-w_0)+\int\limits_0^1 (1-t)(u-w_0)^{T}\nabla^2\Gamma(w_0+t(u-w_0))(u-w_0)\,dt.
\end{align}

By definition of $w_0$, we must have $\Gamma(w_0)=0$ and $\nabla\Gamma(w_0)(u-w_0)\geq0$.

Note that $\nabla^2\Gamma=(1-a_1)\nabla^2\eta$. Thus, by strict convexity of $\eta$ and because $0<a_1<1$, we have $\nabla^2\Gamma\geq cI$ for some constant $c>0$.

We then calculate,
\begin{align}
\int\limits_0^1 (1-t)(u-w_0)^{T}\nabla^2\Gamma(w_0+t(u-w_0))(u-w_0)\,dt\\
\geq \int\limits_0^{.5} (1-t)(u-w_0)^{T}\nabla^2\Gamma(w_0+t(u-w_0))(u-w_0)\,dt,
\shortintertext{where we have changed the limits of integration. Continuing,}
\geq .5 c \abs{u-w_0}^2 \geq .5 c \gamma_0^2,
\end{align}
where the last inequality comes from $\mbox{dist}(u,\{w|\eta(w|u_L)\leq a_1\eta(w|u_R)\})\geq \gamma_0$. This proves \eqref{bound_on_inf}. 

We choose 
\begin{align}\label{epsilon_0_def}
\gamma_0\coloneqq \frac{c_1}{2L_*},
\end{align}
where $c_1$ is from \Cref{dissipation_negative_theorem} and  $L_*$ is the Lipschitz constant of the map
\begin{align}
(u,u_L,u_R)\mapsto a\big(q(u;u_R)-\lambda_1(u)\eta(u|u_R)\big)-q(u;u_L)+\lambda_1(u)\eta(u|u_L).
\end{align}

\uline{Step 2}

Define
\begin{align}\label{V_1_def_RP}
V_1(u)\coloneqq \lambda_{1}(u)-C_{*,1}\mathbbm{1}_{\{u|a_1\eta(u|u_{R,1})<\eta(u|u_{L,1})\}}(u),
\end{align}
where $C_{*,1}>0$ is a large constant, which we can pick to be
\begin{align}\label{C_star_def}
C_{*,1}\coloneqq \frac{1}{c_4\gamma_0^2}\Bigg(\sup_{u,u_L,u_R\in B_B(0)}\abs{aq(u;u_R)-q(u;u_L)}+1\Bigg) + 2\sup_{u\in B_B(0)}\abs{\lambda_1(u)},
\end{align}
where $c_4$ is from \eqref{bound_on_inf}.

We solve the following ODE in the sense of Filippov flows,
\begin{align}
  \begin{cases}\label{ODE}
   \dot{h}_1(t)=V(u(h_1(t),t))\\
   h_1(0)=0,
  \end{cases}
\end{align}

The existence of such an $h$ comes from the following lemma,
\begin{lemma}[Existence and ordering of Filippov flows]\label{Filippov_existence_RP}
For $i=1,2$ let $V_i(u,t):\mathbb{R}^n \times [0,\infty)\to\mathbb{R}$ be bounded on $\mathbb{R}^n \times [0,\infty)$, upper semi-continuous in $u$, and measurable in $t$.  Let $u$ be a weak solution to \eqref{system}, entropic for the entropy $\eta$, and that takes values in a compact set $K$. Assume also that $u$ verifies the strong trace property (\Cref{strong_trace_definition}). Let $x_0\in\mathbb{R}$. Then for $i=1,2$ we can solve 
\begin{align}
  \begin{cases}\label{ODE}
   \dot{g}_i(t)=V_i(u(g_i(t),t),t)\\
   g_i(0)=x_0,
  \end{cases}
\end{align}
in the Filippov sense. That is, there exist Lipschitz functions $g_i:[0,\infty)\to\mathbb{R}$ such that
\begin{align}
\mbox{Lip}[g_i]\leq \norm{V_i}_{L^\infty},\label{fact1}\\
g_i(0)=x_0,\label{fact2}\\
\dot{g}_i(t)\in I[V(u^i_+,t),V(u^i_-,t)],\label{fact3}
\end{align}
for almost every $t$, where $u^i_\pm\coloneqq u(g_i(t)\pm,t)$ and $I[a,b]$ denotes the closed interval with endpoints $a$ and $b$. 

Moreover, for almost every $t$,
\begin{align}
f(u^i_+)-f(u^i_-)=\dot{g}_i(u^i_+-u^i_-),\label{fact4}\\
q(u^i_+)-q(u^i_-)\leq\dot{g}_i(\eta(u^i_+)-\eta(u^i_-)),\label{fact5}
\end{align}
which means that for almost every $t$, either $(u^i_+,u^i_-,\dot{g}_i)$ is an entropic shock (for $\eta$) or $u^i_+=u^i_-$.

Furthermore, if there exists $\mu>0$ such that for all $v\in K$ we have 
\begin{align}\label{flow_GAP}
V_2(v)-V_1(v)\geq \mu,
\end{align}
then $g_1$ and $g_2$ satisfy 
\begin{align}\label{Filippov_ordering_RP}
g_2(t)\geq g_1(t) \mbox{ for all } t\in [0,T).
\end{align}
\end{lemma}

The proof of \eqref{fact1}, \eqref{fact2}, and \eqref{fact3} is very similar to the proof of Proposition 1 in \cite{Leger2011}. 

It is well known that \eqref{fact4} and \eqref{fact5} are true for any Lipschitz continuous function $g:[0,\infty)\to\mathbb{R}$ when $u$ is BV. When instead $u$ is only known to have strong traces (\Cref{strong_trace_definition}), then \eqref{fact4} and \eqref{fact5} are given in Lemma 6 in \cite{Leger2011}. We do not prove \eqref{fact4} and \eqref{fact5} here; their proof is in the appendix in \cite{Leger2011}.

The result \eqref{Filippov_ordering_RP} is a new result about Filippov flows novel to this article.

The proof of \eqref{Filippov_ordering_RP} is in \Cref{Filippov_existence_RP_section}. Moreover, for completeness, the proofs of  \eqref{fact1}, \eqref{fact2} and \eqref{fact3} are also in \Cref{Filippov_existence_RP_section}.

Note that $V_1$ (see \eqref{V_1_def_RP}) is upper semi-continuous in $u$ because indicator functions of open sets are lower semi-continuous and the negative of a lower semi-continuous function is upper semi-continuous.

\uline{Step 3}

Let  $u^1_{\pm}\coloneqq u(u(h_1(t)\pm,t)$.

Note that by \Cref{Filippov_existence_RP}, 
\begin{align}\label{where_h_lives}
\dot{h}_1(t)\in I \Bigg[\lambda_{1}(u^1_+)-C_{*,1}\mathbbm{1}_{\{u|a_1\eta(u|u_{R,1})<\eta(u|u_{L,1})\}}(u^1_+),\\
\lambda_{1}(u^1_-)-C_{*,1}\mathbbm{1}_{\{u|a_1\eta(u|u_{R,1})<\eta(u|u_{L,1})\}}(u^1_-)\Bigg].
\end{align}

We are now ready to show \eqref{dissipation_negative_claim}. 

For each fixed time $t$, we have 4 cases to consider to prove \eqref{dissipation_negative_claim}:\newline
\emph{Case 1}
\begin{align}
a_1\eta(u^1_-|u_{R,1})<\eta(u^1_-|u_{L,1}),\\
a_1\eta(u^1_+|u_{R,1})<\eta(u^1_+|u_{L,1}).
\end{align}
\emph{Case 2}
\begin{align}
a_1\eta(u^1_-|u_{R,1})<\eta(u^1_-|u_{L,1}),\\
a_1\eta(u^1_+|u_{R,1})\geq\eta(u^1_+|u_{L,1}).
\end{align}
\emph{Case 3}
\begin{align}
a_1\eta(u^1_-|u_{R,1})\geq\eta(u^1_-|u_{L,1}),\\
a_1\eta(u^1_+|u_{R,1})<\eta(u^1_+|u_{L,1}).
\end{align}
\emph{Case 4}
\begin{align}
a_1\eta(u^1_-|u_{R,1})\geq\eta(u^1_-|u_{L,1}),\\
a_1\eta(u^1_+|u_{R,1})\geq\eta(u^1_+|u_{L,1}).
\end{align}

Note that we allow for $u^1_+=u^1_-$.

We start with 

\emph{Case 1}

In this case, by \eqref{fact3}, \eqref{C_star_def}, and \eqref{where_h_lives} we know that 
\begin{equation}
\begin{aligned}\label{control_h}
\dot{h}_1(t)\leq -\frac{1}{c_4\gamma_0^2}\Bigg(\sup_{u,u_L,u_R\in B_B(0)}\abs{aq(u;u_R)-q(u;u_L)}+1\Bigg)-\sup_{u\in B_B(0)}\abs{\lambda_1(u)}
\\
<\inf_{u\in B_B(0)} \lambda_1(u).
\end{aligned} 
\end{equation}

If $u^1_+\neq u^1_-$, then we have \eqref{fact4} and \eqref{fact5}. But then \eqref{control_h} contradicts $(\mathcal{H}2)$. Thus, $u^1_+= u^1_-$.

Let $v\coloneqq u^1_+=u^1_-$.

If $\mbox{dist}(v,\{w|\eta(w|u_{L,1})\leq a_1\eta(w|u_{R,1})\})\geq \gamma_0$, then

\begin{equation}
\begin{aligned}\label{dissipation_negative_claim_proof_case1_1}
&a\bigg(q(u^1_+;u_{R,1})-\dot{h}_1(t)\eta(u^1_+|u_{R,1})\bigg)-q(u^1_-;u_{L,1})+\dot{h}_1(t)\eta(u^1_-|u_{L,1}) \\
&\hspace{1in}=a\bigg(q(v;u_{R,1})-\dot{h}_1(t)\eta(v|u_{R,1})\bigg)-q(v;u_{L,1})+\dot{h}_1(t)\eta(v|u_{L,1})\\
&\hspace{1in}=aq(v;\bar{u}_+(t))-q(v;\bar{u}_-(t))-\dot{h}(t)\big(a\eta(v|\bar{u}_+(t))-\eta(v|\bar{u}_-(t))\big)\\
&\hspace{1in}\leq -1,
\end{aligned}
\end{equation}
because of \eqref{control_h} and \eqref{bound_on_inf}. Because the term $\abs{\sigma^1(u_{L,1},u_{R,1})-\dot{h}_1(t)}^2$ on the right hand side of \eqref{dissipation_negative_claim} is bounded due to \eqref{fact1}, we have proven \eqref{dissipation_negative_claim} by choosing $c$ sufficiently small.

If on the other hand, $\mbox{dist}(v,\{w|\eta(w|u_{L,1})\leq a_1\eta(w|u_{R,1})\})< \gamma_0$, then

\begin{equation}
\begin{aligned}\label{dissipation_negative_claim_proof_case1_2}
&a\bigg(q(u^1_+;u_{R,1})-\dot{h}_1(t)\eta(u^1_+|u_{R,1})\bigg)-q(u^1_-;u_{L,1})+\dot{h}_1(t)\eta(u^1_-|u_{L,1}) \\
&\hspace{1in}=a\bigg(q(v;u_{R,1})-\dot{h}_1(t)\eta(v|u_{R,1})\bigg)-q(v;u_{L,1})+\dot{h}_1(t)\eta(v|u_{L,1}) \\
&\hspace{1in}=aq(v;\bar{u}_+(t))-q(v;\bar{u}_-(t))-\dot{h}(t)\big(a\eta(v|\bar{u}_+(t))-\eta(v|\bar{u}_-(t))\big)\\
&\hspace{1in}\leq a\bigg(q(v;u_{R,1})-\lambda_1(v)\eta(v|u_{R,1})\bigg)-q(v;u_{L,1})+\lambda_1(v)\eta(v|u_{L,1}),\\
&\mbox{because $\eta(v|u_{L,1})-a_1\eta(v|u_{R,1})\geq 0$ and $\dot{h}_1\leq -\sup_{u\in B_B(0)}\abs{\lambda_1(u)}$. Continuing,}\\
&\mbox{we get}
\\
&\hspace{1in}\leq -\frac{1}{2}c_1,
\end{aligned}
\end{equation}
from \eqref{dissipation_negative_boundary_of_convex_set}, the definition of $\gamma_0$ \eqref{epsilon_0_def}, the assumption that $\mbox{dist}(v,\{w|\eta(w|u_{L,1})\leq a_1\eta(w|u_{R,1})\})< \gamma_0$ and the assumption that  $r_1\geq \rho$. Again because the term $\abs{\sigma^1(u_{L,1},u_{R,1})-\dot{h}_1(t)}^2$ on the right hand side of \eqref{dissipation_negative_claim} is bounded due to \eqref{fact1}, we have proven \eqref{dissipation_negative_claim} by choosing $c$ sufficiently small. Note $c$ will depend on $\rho$.

\emph{Case 2}

In this case, we must have $u^1_-\neq u^1_+$. Recall also that \eqref{system} is hyperbolic. Furthermore, we have from \eqref{fact3} that $\dot{h}_1\in \Bigg[-\frac{1}{c_4\gamma_0^2}\Bigg(\sup_{u,u_L,u_R\in B_B(0)}\abs{aq(u;u_R)-q(u;u_L)}+1\Bigg)-\sup_{u\in B_B(0)}\abs{\lambda_1(u)},\lambda_1(u^1_+)\Bigg]$. However, this implies that $(u^1_+,u^1_-,\dot{h}_1)$ is a right 1-contact discontinuity (see \cite[p.~274]{dafermos_big_book}). This contradicts the hypothesis $(\mathcal{H}2)$ on the shock $(u_+,u_-,\dot{h})$, which is entropic for $\eta$ because of \eqref{fact4} and \eqref{fact5}. The hypothesis $(\mathcal{H}2)$ forbids right 1-contact discontinuities. Thus, we conclude that this case (\emph{Case 2}) cannot actually occur.

\emph{Case 3}

In this case, we have from \eqref{fact3} that 
\begin{align}
\dot{h}_1\in \Bigg[-\frac{1}{c_4\gamma_0^2}\Bigg(\sup_{u,u_L,u_R\in B_B(0)}\abs{aq(u;u_R)-q(u;u_L)}+1\Bigg)-\sup_{u\in B_B(0)}\abs{\lambda_1(u)},\lambda_1(u^1_-)\Bigg].
\end{align}
By the hypothesis $(\mathcal{H}3)$, along with \eqref{fact4}, \eqref{fact5}, we have that $(u^1_+,u^1_-,\dot{h}_1)$ must be a 1-shock. Also, $u^1_-$ verifies $a_1\eta(u^1_-|u_{R,1})\geq\eta(u^1_-|u_{L,1})$. Thus, we can apply \Cref{dissipation_negative_theorem}. Recall that $r_1>\rho$ (see \eqref{gap_local_case}). We receive  \eqref{dissipation_negative_claim}.

\emph{Case 4}

In this case, we have from \eqref{fact3} that $\dot{h}_1\in I[\lambda_1(u^1_+),\lambda_1(u^1_-)]$. Then, by the hypothesis $(\mathcal{H}2)$, along with \eqref{fact4}, \eqref{fact5}, we know that we cannot have
\begin{align}\label{this_would_imply_bad}
I[\lambda_1(u_+),\lambda_1(u_-)]=(\lambda_1(u_-),\lambda_1(u_+))
\end{align}
because then \eqref{this_would_imply_bad} would imply that $(u^1_+,u^1_-,\dot{h}_1)$ is a right 1-contact discontinuity. However, $(\mathcal{H}2)$ prevents right 1-contact discontinuities. Recall $(\mathcal{H}3)$. We conclude that $(u^1_+,u^1_-,\dot{h}_1)$ is a 1-shock. Moreover, $u^1_-$ verifies $a_1\eta(u^1_-|u_{R,1})\geq\eta(u^1_-|u_{L,1})$. We can now apply \Cref{dissipation_negative_theorem}. Recall that $r_1>\rho$ (see \eqref{gap_local_case}). This gives \eqref{dissipation_negative_claim}.

\uline{Proof of \eqref{dissipation_negative_claim_n}}

To prove \eqref{dissipation_negative_claim_n}, note that if $v(x,t)$ solves \eqref{system}, then $v(-x,t)$ will solve 
\begin{align}\label{system_rp_backwards}
v_t+(-f(v))_x=0,
\end{align}
where we have replaced the flux $f$ with $-f$.

The $n\textsuperscript{th}$ characteristic family of \eqref{system} corresponds to the first characteristic family of \eqref{system_rp_backwards}. Thus to prove \eqref{dissipation_negative_claim_n} we simply apply \eqref{dissipation_negative_claim} to the system \eqref{system_rp_backwards}.

Define 
\begin{align}\label{V_n_def_RP}
V_n(u)\coloneqq \lambda_{n}(u)+C_{*,n}\mathbbm{1}_{\{u|a_n\eta(u|u_{L,n})<\eta(u|u_{R,n})\}}(u).
\end{align}

Note that the shift function $h_n$ (from \eqref{dissipation_negative_claim_n}) will solve the following ODE in the sense of Filippov flows,
\begin{align}
  \begin{cases}\label{ODE_n}
   \dot{h}_n(t)=V(u(h_n(t),t))\\
   h_n(0)=0,
  \end{cases}
\end{align}
for a large constant $C_{*,n}>0$ and where $\lambda_n$ as usual refers to the $n\textsuperscript{th}$ characteristic family of \eqref{system}.

Let $K$ be a compact set which contains the range of $u$ (note by assumption $u$ is bounded). Then due to the strict hyperbolicity of \eqref{system}, there is $\theta>0$ such that
\begin{align}\label{lambda_gap_RP}
\lambda_n(v)-\lambda_1(v)\geq \theta
\end{align}
for all $v\in K$. 
 
Then, \eqref{lambda_gap_RP} along with \eqref{V_1_def_RP} and \eqref{V_n_def_RP} imply that $V_1$ and $V_n$ satisfy \eqref{flow_GAP} for some $\mu$.

Then \eqref{Filippov_ordering_RP} implies \eqref{ordering_h_1_h_n_RP}.

This completes the proof of \Cref{systems_entropy_dissipation_room}.

\subsection{Proof of \Cref{Filippov_existence_RP}}\label{Filippov_existence_RP_section}

\hfill \break

\emph{Proof of \eqref{fact1}, \eqref{fact2}, and \eqref{fact3}}

The following proof of \eqref{fact1}, \eqref{fact2}, and \eqref{fact3} is based on the proof of Proposition 1 in \cite{Leger2011}, the proof of Lemma 2.2 in \cite{serre_vasseur}, and the proof of Lemma 3.5 in \cite{2017arXiv170905610K}. We do not prove \eqref{fact4} or \eqref{fact5} here; these properties are in Lemma 6 in \cite{Leger2011}, and their proofs are in the appendix in \cite{Leger2011}.

For $i=1,2$ define
\begin{align}\label{mollified_v_flow}
v_{i,n}(x,t)\coloneqq \int\limits_0^1 V_i\bigg(u(x+\frac{y}{n},t),t\bigg)\,dy.
\end{align}

Let $g_{i,n}$ be the solution to the ODE:
\begin{align}
  \begin{cases}\label{n_ode_RP}
   \dot{g}_{i,n}(t)=v_{i,n}(g_{i,n}(t),t),\mbox{ for }t>0\\
   g_{i,n}(0)=x_0.
  \end{cases}
\end{align}

The $v_{i,n}$ are uniformly bounded in $n$ because by assumption $V_i$ is bounded (\hspace{.07cm}$\norm{v_{i,n}}_{L^\infty}\leq \norm{V_i}_{L^\infty}$). The $v_{i,n}$ are measurable in $t$, and due to the mollification by $\frac{1}{n}$ are also Lipschitz continuous in $x$. Thus \eqref{n_ode_RP} has a unique solution in the sense of Carath\'eodory.

The $g_{i,n}$ are Lipschitz continuous with Lipschitz constants uniform in $n$, due to the $v_{i,n}$ being uniformly bounded in $n$. Thus, by Arzel\`a--Ascoli the $g_{i,n}$ converge in $C^0(0,T)$ for any fixed $T>0$ to a Lipschitz continuous function $g_i$ (passing to a subsequence if necessary). Note that $\dot{g}_{i,n}$ converges in $L^\infty$ weak* to $\dot{g}_i$.

We define
\begin{align}
V_{i,\mbox{max}}(t)\coloneqq \max\{V(u^i_-,t),V(u^i_+,t)\},\\
V_{i,\mbox{min}}(t)\coloneqq \min\{V(u^i_-,t),V(u^i_+,t)\},
\end{align}
where $u^i_\pm \coloneqq u(g_i(t)\pm,t)$.

To show \eqref{fact3}, we will first prove that for almost every $t>0$
\begin{align}
\lim_{n\to\infty}[\dot{g}_{i,n}(t)-V_{i,\mbox{max}}(t)]_+=0,\label{limit_1}\\
\lim_{n\to\infty}[V_{i,\mbox{min}}(t)-\dot{g}_{i,n}(t)]_+=0,\label{limit_2}
\end{align}
where $[\hspace{.1cm}\cdot\hspace{.1cm}]_+\coloneqq\max(0,\cdot)$.

The proofs of \eqref{limit_1} and \eqref{limit_2} are similar; we only show the first one.

\begin{align}
[\dot{g}_{i,n}(t)-V_{i,\mbox{max}}(t)]_+\\
=\Bigg[\int\limits_0^1 V_i\bigg(u(g_{i,n}(t)+\frac{y}{n},t),t\bigg)\,dy-V_{i,\mbox{max}}(t)\Bigg]_+\\
=\Bigg[\int\limits_0^1 V_i\bigg(u(g_{i,n}(t)+\frac{y}{n},t),t\bigg)-V_{i,\mbox{max}}(t)\,dy\Bigg]_+\\
\leq\int\limits_0^1 \Big[V_i\bigg(u(g_{i,n}(t)+\frac{y}{n},t),t\bigg)-V_{i,\mbox{max}}(t)\Big]_+\,dy\\
\leq\esssup_{y\in(0,\frac{1}{n})} \Big[V_i\bigg(u(g_{i,n}(t)+y,t),t\bigg)-V_{i,\mbox{max}}(t)\Big]_+\\
\leq\esssup_{y\in(-\epsilon_{i,n},\epsilon_{i,n})} \Big[V_i\bigg(u(g_i(t)+y,t),t\bigg)-V_{i,\mbox{max}}(t)\Big]_+,\label{last_ineq_Filippov}
\end{align}
where $\epsilon_{i,n}\coloneqq \abs{g_{i,n}(t)-g_i(t)}+\frac{1}{n}$. Note $\epsilon_{i,n}\to0^+$.

Fix a $t\geq0$ such that $u$ has a strong trace in the sense of \Cref{strong_trace_definition}. Then because the map $u\mapsto V_i(u,t)$ is upper semi-continuous,
\begin{align}\label{esssuplim_is_zero}
\lim_{n\to\infty}\esssup_{y\in(0,\frac{1}{n})} \Big[V_i\bigg(u(g_i(t)\pm y,t),t\bigg)-V_i\big(u^i_\pm,t\big)\Big]_+=0,
\end{align}
where $u^i_\pm \coloneqq u(g_i(t)\pm,t)$. Recall that the map $u\mapsto V_i(u,t)$ being upper semi-continuous at the point $u_0$ means that 
\begin{align}
\limsup_{u\to u_0} V_i(u,t) \leq V_i(u_0,t).
\end{align}

From \eqref{esssuplim_is_zero}, we get
\begin{align}\label{esssuplim_is_zero2}
\lim_{n\to\infty}\esssup_{y\in(0,\frac{1}{n})} \Big[V_i\bigg(u(g_i(t)\pm y,t),t\bigg)-V_{i,\mbox{max}}(t)\Big]_+=0.
\end{align}

We can control \eqref{last_ineq_Filippov} from above by the quantity
\begin{equation}
\begin{aligned}\label{esssuplim_is_zero3}
\esssup_{y\in(-\epsilon_{i,n},0)} \Big[V_i\bigg(u(g_i(t)+ y,t),t\bigg)-V_{i,\mbox{max}}(t)\Big]_++\\
\esssup_{y\in(0,\epsilon_{i,n})} \Big[V_i\bigg(u(g_i(t)+ y,t),t\bigg)-V_{i,\mbox{max}}(t)\Big]_+.
\end{aligned}
\end{equation}

By \eqref{esssuplim_is_zero2}, we have that \eqref{esssuplim_is_zero3} goes to $0$ as $n\to\infty$. This proves \eqref{limit_1}.

Recall that $\dot{g}_{i,n}$ converges in $L^\infty$ weak* to $\dot{g}_i$. Thus, due to the convexity of the function $[\hspace{.1cm}\cdot\hspace{.1cm}]_+$,
\begin{align}
\int\limits_0^T[\dot{g}_i(t)-V_{i,\mbox{max}}(t)]_+\,dt\leq \liminf_{n\to\infty}\int\limits_0^T[\dot{g}_{i,n}(t)-V_{i,\mbox{max}}(t)]_+\,dt.
\end{align}

By the dominated convergence theorem and \eqref{limit_1},
\begin{align}
\liminf_{n\to\infty}\int\limits_0^T[\dot{g}_{i,n}(t)-V_{i,\mbox{max}}(t)]_+\,dt=0.
\end{align}

We conclude,
\begin{align}
\int\limits_0^T[\dot{g}_i(t)-V_{i,\mbox{max}}(t)]_+\,dt=0.
\end{align}

From a similar argument,
\begin{align}
\int\limits_0^T[V_{i,\mbox{min}}(t)-\dot{g}_i(t)]_+\,dt=0.
\end{align}

This proves \eqref{fact3}.

\emph{Proof of \eqref{Filippov_ordering_RP}}

\begin{figure}[tb]
      \includegraphics[width=\textwidth]{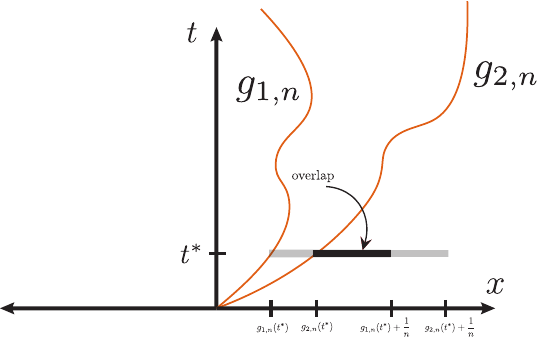}
  \caption{The idea for the proof of \eqref{Filippov_ordering_RP}.}\label{two_flows_figure}
\end{figure}

Let us first explain the idea behind the proof of \eqref{Filippov_ordering_RP}. We use the fact that, for a fixed $t$, according to \eqref{mollified_v_flow} and \eqref{n_ode_RP}, the value of $\dot{g}_{i,n}(t)$ is based on the value of $u(x,t)$ for $x\in[g_{i,n}(t),g_{i,n}(t)+\frac{1}{n}]$. Then, if the values of $g_{1,n}(t)$ and $g_{2,n}(t)$ are close enough together (see \eqref{condition_on_t_RP_flow} below), the domain of $u(\cdot,t)$ used to calculate $\dot{g}_{1,n}(t)$ and the domain of $u(\cdot,t)$ used to calculate $\dot{g}_{2 ,n}(t)$  (according to \eqref{mollified_v_flow} and \eqref{n_ode_RP}) will have some overlap. On this overlap, the estimate \eqref{flow_GAP} says that 
\begin{align}\label{flow_GAP_used_in_proof36}
V_2(u(\cdot,t))-V_1(u(\cdot,t))>\mu.
\end{align}
Thus, when the value of $g_{1,n}(t)$ and $g_{2,n}(t)$ are close enough together, the estimate \eqref{flow_GAP_used_in_proof36} allows us to compensate for the lack of control we have for the parts of the domain of  $u(\cdot,t)$ which are not overlapping, and we find that whenever $g_{1,n}(t)$ and $g_{2,n}(t)$ are close enough together, the difference $\dot{g}_{2,n}-\dot{g}_{1,n}$ must be strictly positive (see \eqref{uniform_control_flow_RP}). This means that  whenever $g_{1,n}$ and $g_{2,n}$ get close together, they start being pushed apart. This, combined with the the identical starting values $g_{2,n}(0)=g_{1,n}(0)=x_0$, yields \eqref{Filippov_ordering_RP} in the $n\to\infty$ limit. See \Cref{two_flows_figure}.

\hfill

We now give the proof.

Fix $n\in\mathbb{N}$. 

Define
\begin{align}\label{RP_M_DEF}
M\coloneqq\max_{i\in\{1,2\}}\norm{V_i}_{L^\infty}.
\end{align}

Assume that for some $t^*$, 
\begin{align}\label{condition_on_t_RP_flow}
\abs{g_{2,n}(t^*)-g_{1,n}(t^*)}<\frac{\mu}{n(\mu+4M)}.
\end{align}

Recall that due to $g_{1,n}$ and $g_{2,n}$ solving \eqref{n_ode_RP} in the sense of Carath\'eodory, for $i=1,2$ they satisfy 
\begin{align}\label{differential_equation_itself_RP}
 \dot{g}_{i,n}(t)=v_{i,n}(g_{i,n}(t),t),
\end{align}
for almost every $t$.

Then if $g_{1,n}$ and $g_{2,n}$ also satisfy the differential equation \eqref{differential_equation_itself_RP} at this time $t^*$, then we have
\begin{equation}
\begin{aligned}\label{string_equals_flow_RP}
&\dot{g}_{2,n}(t^*)-\dot{g}_{1,n}(t^*)
\\
&\hspace{1in}=v_{2,n}(g_{2,n}(t^*),t^*)-v_{1,n}(g_{1,n}(t^*),t^*)
\\
&\hspace{1in}=\int\limits_0^1 V_2\bigg(u(g_{2,n}(t^*)+\frac{y}{n},t^*),t^*\bigg)\,dy
-
\int\limits_0^1 V_1\bigg(u(g_{1,n}(t^*)+\frac{y}{n},t^*),t^*\bigg)\,dy.
\end{aligned}
\end{equation}

Then from \eqref{flow_GAP}, \eqref{RP_M_DEF} and \eqref{condition_on_t_RP_flow}, we have
\begin{equation}
\begin{aligned}\label{uniform_control_flow_RP}
&\int\limits_0^1 V_2\bigg(u(g_{2,n}(t^*)+\frac{y}{n},t^*),t^*\bigg)\,dy
-
\int\limits_0^1 V_1\bigg(u(g_{1,n}(t^*)+\frac{y}{n},t^*),t^*\bigg)\,dy
\\
&\hspace{3.5in}\geq \frac{2M\mu}{\mu+4M}.
\end{aligned}
\end{equation}

Thus,  from \eqref{string_equals_flow_RP} and \eqref{uniform_control_flow_RP} and the fundamental theorem of calculus for $W^{1,1}_{\text{loc}}$ functions we get that for any $t_1,t_2\in[0,T)$ verifying $t_1< t_2$ and such that condition \eqref{condition_on_t_RP_flow} holds for all $t^*\in[t_1,t_2]$,
\begin{align}\label{ordering_t1_t2_flow_RP}
g_{2,n}(t_1)-g_{1,n}(t_1)<g_{2,n}(t_2)-g_{1,n}(t_2).
\end{align}
Because $g_{2,n}(0)=g_{1,n}(0)=x_0$, \eqref{ordering_t1_t2_flow_RP} implies that 
\begin{align}\label{ordering_all_t_flow_RP_n}
g_{1,n}(t)\leq g_{2,n}(t)
\end{align}
for all $t\in[0,T)$.

Thus in the $n\to\infty$ limit, from \eqref{ordering_all_t_flow_RP_n} we in fact get \eqref{Filippov_ordering_RP}.


\section{Proofs of \Cref{RP_theorem_1} and \Cref{RP_theorem_3}}\label{proofs_of_multiple_stability_theorems}
\subsection{\Cref{RP_theorem_1}: $L^2$ Stability for the Riemann Problem with Extremal Shocks Verifying Strong Form of Lax's E-condition}
\begin{theorem}[$L^2$ Stability for the Riemann Problem with Extremal Shocks Verifying Strong Form of Lax's E-condition]
\label{RP_theorem_1}
Fix $T>0$. Assume $u,\bar{v}\in L^\infty(\mathbb{R}\times[0,T))$ are solutions to the system \eqref{system}. Assume that $u$ and $\bar{v}$ are entropic for the entropy $\eta$.  Further, assume that $u$ has strong traces (\Cref{strong_trace_definition}). 

Assume also that $\bar{v}$ is a solution to the Riemann problem \eqref{Riemann_problem} and that $\bar{v}$ has the form \eqref{standard_form_solution_RP}.  If $\bar{v}$ contains a 1-shock, assume the hypotheses $(\mathcal{H})$ hold. Likewise, if $\bar{v}$ contains an n-shock, assume the hypotheses $(\mathcal{H})^*$ hold.

Assume that $\bar{v}$ contains at least one rarefaction wave, and if there are any shocks in $\bar{v}$ they are either a 1-shock verifying  \eqref{lax_strong} or an n-shock verifying \eqref{lax_strong}.

Also assume \eqref{RHS_good_sign_rarefaction} holds. Further, assume the system \eqref{system} has at least two conserved quantities ($n\geq2$). 

Then there exists $\Psi_{\bar{v}}$ with Property $(\mathcal{D})$ and verifying the following stability estimate:
\begin{align}\label{RP_lax_strong_stability_theorem_result}
\int\limits_{-R}^{R}\abs{u(x,t_0)-\Psi_{\bar{v}}(x,t_0)}^2\,dx\leq\mu_1\int\limits_{-R-rt_0}^{R+rt_0}\abs{u^0(x)-\bar{v}(x,0)}^2\,dx,
\end{align}
for all $t_0,R>0$ verifying $t_0\in(0,R)$ and
\begin{align}\label{R_sufficiently_large_RP_theorem_statement}
R>\max_{i}\{\mbox{Lip}[h_i]\}t_0,
\end{align}
where the max runs over the i-shock families contained in $\bar{v}$ (1-shocks and/or n-shocks) and the $h_i$ are in the context of Property $(\mathcal{D})$.

We also have the following $L^2$-type control on the shift functions $h_i$:
\begin{align}\label{control_shifts_two_shifts204}
\int\limits_0^{t_0} \sum_{i}\abs{\sigma^i(\bar{v}_i,\bar{v}_{i+1})-\dot{h}_i(t)}^2\,dt\leq \mu_2\int\limits_{-R-rt_0}^{R+rt_0}\abs{u^0(x)-\bar{v}(x,0)}^2\,dx,
\end{align}
where the sum runs over the i-shock families contained in $\bar{v}$ (1-shocks and/or n-shocks). 

If $\bar{v}$ contains an i-shock, the constants $\mu_1,\mu_2>0$ depend on $\norm{u}_{L^\infty}$, $\abs{\bar{v}_i},\abs{\bar{v}_{i+1}}$, and $\abs{\bar{v}_i-\bar{v}_{i+1}}$. Further, $\mu_1$, and $\mu_2$ depend on bounds on the second derivative of $\eta$ on the range of $u$ and $\bar{v}$. In addition, if $\bar{v}$ contains an i-shock, $\mu_1$ depends on $\sup\norm{\nabla \lambda_i}$ (where the supremum runs over the range of $u$ and $\bar{v}$), and  $(\lambda_{i+1}(\bar{v}_{i+1})-\lambda_i(\bar{v}_i))$. 
\end{theorem}

\begin{proof}

We assume that in $\bar{v}$ there is both a 1-shock and an n-shock (in addition to a rarefaction fan). The three cases when there is only a 1-shock, only an n-shock, or no shocks at all are all very similar and are left to the reader. 

We first focus on the shock connecting $\bar{v}_1$ to $\bar{v}_2$. Label the shock speed $\sigma^1(\bar{v}_1,\bar{v}_{2})$. 

Let $j$ be such that the leftward most rarefaction wave in $\bar{v}$ is a j-rarefaction wave.

Then the j-rarefaction wave must be joining $\bar{v}_j=\bar{v}_2$ and $\bar{v}_{j+1}$. We can write this j-rarefaction wave as,
\begin{align}
\begin{cases}\label{2_rarefaction_joining_RP}
\bar{v}_j & \text{if } \frac{x}{t}<\lambda_j(\bar{v}_j)\\ 
V_j(\frac{x}{t})& \text{if } \lambda_j(\bar{v}_j)<\frac{x}{t}<\lambda_j(\bar{v}_{j+1})\\ 
\bar{v}_{j+1} & \text{if }\lambda_j(\bar{v}_{j+1})<\frac{x}{t},
\end{cases}
\end{align}
for a function $V_j:\mathbb{R}\to\mathbb{R}^n$. 

From \eqref{lax_strong}, we get that $\lambda_1(\bar{v}_1)<\lambda_2(\bar{v}_2)$. Then, from strict hyperbolicity of the system \eqref{system}, we get $\lambda_1(\bar{v}_1)<\lambda_2(\bar{v}_2)\leq\lambda_j(\bar{v}_2)$ (recall it is possible $j=2$).   Define 
\begin{equation}
\begin{aligned}\label{defs_Riemann_problem}
L&\coloneqq \sup_{\abs{u}\leq B}\norm{\nabla \lambda_1(u)},\\
\epsilon_0&\coloneqq \frac{\lambda_2(\bar{v}_2)-\lambda_1(\bar{v}_1)}{2L},
\end{aligned}
\end{equation}
where $B$ verifies $\norm{u}_{L^\infty},\norm{\bar{v}}_{L^\infty}\leq B$. Note $L$ exists by the remarks after the hypotheses $(\mathcal{H})$ and $(\mathcal{H})^*$.

From \Cref{systems_entropy_dissipation_room}, we can find a positive $a_1$ such that 
\begin{align}\label{condition_on_epsilon_0_riemann_problem}
a_1<\frac{\epsilon_0^2}{C},
\end{align}
where $C$ is the constant from \Cref{a_cond_lemma_itself}, and 
from \eqref{dissipation_negative_claim}, we have a shift function $h_1:[0,T)\to\mathbb{R}$ such that $h_1(0)=0$ and
\begin{equation}
\begin{aligned}\label{dissipation_negative_claim_Riemann_problem}
a_1\big(q(u_+^1;\bar{v}_2)&-\dot{h}_1(t)\eta(u_+^1|\bar{v}_2)\big)-q(u_-^1;\bar{v}_1)+\dot{h}_1(t)\eta(u_-^1|\bar{v}_1) \leq \\
& -c_1 \abs{\sigma^1(\bar{v}_1,\bar{v}_{2})-\dot{h}_1(t)}^2,
\end{aligned}
\end{equation}
for all $t\in[0,T)$ and where $u_{\pm}^1\coloneqq u(u(h_1(t)\pm,t)$. 

Note that from \Cref{a_cond_lemma_itself} and \eqref{condition_on_epsilon_0_riemann_problem} we know that $\{u | \eta(u|\bar{v}_1)\leq a_1\eta(u|\bar{v}_2)\} \subset B_{\epsilon_0}(\bar{v}_1)$. 

Furthermore, for each $t\in[0,T)$ either $\dot{h}_1(t)<\inf\lambda_1$ or $(u_+^1,u_-^1,\dot{h}_1)$ is a 1-shock with $u_-^1\in \{u | \eta(u|\bar{v}_1)\leq a_1\eta(u|\bar{v}_2)\}$ (possibly $u^1_+=u^1_-$ and $\dot{h}_1=\lambda_1(u^1_\pm)$).  From   \eqref{defs_Riemann_problem} and \eqref{condition_on_epsilon_0_riemann_problem}, we have that 
\begin{align}\label{short_computation_using_lambda_Lip_Riemann_problem}
\lambda_1(u_-^1)\leq \lambda_1(\bar{v}_1)+\frac{\lambda_2(\bar{v}_2)-\lambda_1(\bar{v}_1)}{2}<\lambda_2(\bar{v}_2).
\end{align}

Then, due to the hypothesis $(\mathcal{H}1)$, $\dot{h}_1(t)\leq \lambda_1(u_-^1)$. Then because of \eqref{short_computation_using_lambda_Lip_Riemann_problem},  
\begin{align}\label{good_ordering_RP_part1}
\dot{h}_1(t)\leq \lambda_1(u_-^1)< \lambda_2(\bar{v}_2)
\end{align} 
for all $t\in[0,T)$. Finally, recalling that $\lambda_2(\bar{v}_2)\leq \lambda_j(\bar{v}_2)$ due to strict hyperbolicity of \eqref{system}, we get
\begin{align}\label{good_ordering_RP}
\dot{h}_1(t)\leq \lambda_1(u_-^1)< \lambda_2(\bar{v}_2)\leq \lambda_j(\bar{v}_2)
\end{align} 
for all $t\in[0,T)$.

We now consider the n-shock connecting $\bar{v}_n$ to $\bar{v}_{n+1}$. 

Let $k$ be such that the rightward most rarefaction wave in the solution $\bar{v}$ is a k-rarefaction wave. 

Note first that the $k$-rarefaction wave joins $\bar{v}_{k}$ and $\bar{v}_{k+1}=\bar{v}_n$. Note $k\leq n-1$. We can write this k-rarefaction wave as,
\begin{align}
\begin{cases}\label{n_1_rarefaction_joining_RP}
\bar{v}_{k} & \text{if } \frac{x}{t}<\lambda_{k}(\bar{v}_{k})\\ 
V_{k}(\frac{x}{t})& \text{if } \lambda_{k}(\bar{v}_{k})<\frac{x}{t}<\lambda_{k}(\bar{v}_n)\\ 
\bar{v}_n & \text{if }\lambda_{k}(\bar{v}_n)<\frac{x}{t},
\end{cases}
\end{align}
for a function $V_{k}:\mathbb{R}\to\mathbb{R}^n$.

Following the same argument as above for the 1-shock, we get from \eqref{dissipation_negative_claim_n}, a function $h_n:[0,T)\to\mathbb{R}$ such that $h_n(0)=0$ and
\begin{equation}
\begin{aligned}\label{dissipation_negative_claim_Riemann_problem_n_shock}
\frac{1}{a_n}\big(q(u_+^n;\bar{v}_{n+1})&-\dot{h}_n(t)\eta(u_+^n|\bar{v}_{n+1})\big)-q(u_-^n;\bar{v}_n)+\dot{h}_n(t)\eta(u_-^n|\bar{v}_n) \leq \\
& -c_n \abs{\sigma^n(\bar{v}_n,\bar{v}_{n+1})-\dot{h}_n(t)}^2,
\end{aligned}
\end{equation}
for all $t\in[0,T)$ and where $u_{\pm}^n\coloneqq u(u(h_n(t)\pm,t)$ and $0<a_n<1$ is a constant.  For each $t\in[0,T)$ either $\dot{h}_n(t)>\sup\lambda_n$ or $(u^n_+,u^n_-,\dot{h}_n)$ is an n-shock with $u^n_+\in \{u | \eta(u|\bar{v}_{n+1})\leq a_n\eta(u|\bar{v}_{n})\}$ (possibly $u^n_+=u^n_-$ and $\dot{h}_n=\lambda_n(u^n_\pm)$).

We get (as an analogue of \eqref{good_ordering_RP}),
\begin{align}\label{good_ordering_RP_n_shock}
\dot{h}_n(t)\geq \lambda_{n}(u_+^n)> \lambda_{n-1}(\bar{v}_n)\geq \lambda_{k}(\bar{v}_n)
\end{align} 
for all $t\in[0,T)$.

Define $\Psi_{\bar{v}}:\mathbb{R}\times[0,T)\to\mathbb{R}^n$,
\begin{align}
\Psi_{\bar{v}}(x,t)\coloneqq
\begin{cases}\label{Psi_def_RP}
\bar{v}_1& \text{if } x<h_1(t)\\
\bar{v}_2& \text{if } h_1(t)<x<\lambda_j(\bar{v}_2)t\\
\bar{v}(x,t)& \text{if } \lambda_j(\bar{v}_2)t<x<\lambda_{k}(\bar{v}_n)t\\
\bar{v}_n& \text{if } \lambda_{k}(\bar{v}_n)t<x<h_n(t)\\
\bar{v}_{n+1}& \text{if } h_n(t)<x.
\end{cases}
\end{align}

Note that $h_1$ and $h_2$ satisfy \eqref{h_verify_1} and \eqref{h_verify_2}, respectively, and $\Psi_{\bar{v}}$ is well-defined, due to \eqref{good_ordering_RP}, \eqref{good_ordering_RP_n_shock}, and the fundamental theorem of calculus for $W^{1,1}_{\text{loc}}$ functions.

Choose 
\begin{align}\label{R_sufficiently_large_RP}
R>\max\{\mbox{Lip}[h_1],\mbox{Lip}[h_n]\}t_0.
\end{align}

Define
\begin{equation}
\begin{aligned}\label{h_defs_RP_1}
h_{\text{left}}(t)\coloneqq -R +r(t-t_0)\\
h_{\text{right}}(t)\coloneqq R -r(t-t_0),
\end{aligned}
\end{equation}
where $r>0$ satisfies
\begin{align}\label{r_def_RP1}
\abs{q(a;b)}\leq r\eta(a|b),
\end{align}
for $a,b$ within the range of $u$ and $\bar{v}$. Note that $r>0$ exists due to $\eta(a|b)$ and $q(a;b)$ both being locally quadratic in $a-b$ and $\eta$ being strictly convex.

Then, we use \Cref{local_entropy_dissipation_rate_systems_RP} three times with $X\equiv0$:

We use \Cref{local_entropy_dissipation_rate_systems_RP} once with $h_{\text{left}}$ and $h_1$ and with the constant function
\begin{align}
\bar{v}_1
\end{align}
playing the role of the function $\bar{u}$ in  \Cref{local_entropy_dissipation_rate_systems_RP}. Note that $h_1(t)-h_{\text{left}}(t)>0$ for all $t$ due to \eqref{R_sufficiently_large_RP}.

We use \Cref{local_entropy_dissipation_rate_systems_RP} again with $h_1$ and $h_n$ and with
\begin{align}
\begin{cases}\label{function_for_second_use_of_local_entropy_dissipation_rate_systems_RP}
\bar{v}_2& \text{if } x<\lambda_2(\bar{v}_2)t\\
\bar{v}(x,t)& \text{if } \lambda_2(\bar{v}_2)t<x<\lambda_{n-1}(\bar{v}_n)t\\
\bar{v}_n& \text{if } \lambda_{n-1}(\bar{v}_n)t<x
\end{cases}
\end{align}
playing the role of $\bar{u}$. 

Note that $h_1(t)-h_n(t)>0$ for all $t>0$, because of \eqref{good_ordering_RP} and \eqref{good_ordering_RP_n_shock}. Note also that the fact that the solution $\bar{v}$ to the Riemann problem \eqref{Riemann_problem} exists means that $\lambda_2(\bar{v}_2)<\lambda_{n-1}(\bar{v}_n)$. Recall also the fundamental theorem of calculus for $W^{1,1}_{\text{loc}}$ functions.

Further, remark that \eqref{function_for_second_use_of_local_entropy_dissipation_rate_systems_RP} is Lipschitz continuous on $\mathbb{R}\times(0,\infty)$, due to the form of the rarefaction waves  \eqref{2_rarefaction_joining_RP} and\eqref{n_1_rarefaction_joining_RP}.

Finally, we use \Cref{local_entropy_dissipation_rate_systems_RP} a third time, with $h_n$ and $h_{\text{right}}$. Note that $h_{\text{right}}(t)-h_n(t)>0$ for all $t$ due to \eqref{R_sufficiently_large_RP}. The constant function
\begin{align}
\bar{v}_{n+1}
\end{align}
plays the role of $\bar{u}$.

We now take a linear combination of the three applications of \Cref{local_entropy_dissipation_rate_systems_RP}. Recall the space derivative of constant states is zero, and otherwise we have \eqref{RHS_good_sign_rarefaction}. This yields,

\begin{equation}
\begin{aligned}\label{local_compatible_dissipation_calc_RP_1}
&\int\limits_{0}^{t_0} \bigg[q(u(h_{\text{left}}(t)+,t);\bar{v}_1)-\dot{h}_{\text{left}}(t)\eta(u(h_{\text{left}}(t)+,t)|\bar{v}_1)
\\
&\hspace{.25in}-\frac{a_1}{a_n}\big(q(u(h_{\text{right}}(t)-,t);\bar{v}_{n+1})-\dot{h}_{\text{right}}(t)\eta(u(h_{\text{right}}(t)-,t)|\bar{v}_{n+1})\big)
\\
&\hspace{.25in}+
\frac{a_1}{a_n}\big(q(u_+^n;\bar{v}_{n+1})-\dot{h}_n(t)\eta(u_+^n|\bar{v}_{n+1})\big)-a_1\big(q(u_-^n;\bar{v}_n)+\dot{h}_n(t)\eta(u_-^n|\bar{v}_n)\big)
\\
&\hspace{.25in}+
a_1\big(q(u_+^1;\bar{v}_2)-\dot{h}_1(t)\eta(u_+^1|\bar{v}_2)\big)-q(u_-^1;\bar{v}_1)+\dot{h}_1(t)\eta(u_-^1|\bar{v}_1) 
\bigg]\,dt
\\
&\hspace{.25in}\geq
\Bigg[\int\limits_{h_{\text{left}}(t_0)}^{h_1(t_0)}\eta(u(x,t_0)|\Psi_{\bar{v}}(x,t_0))\,dx+a_1\int\limits_{h_1(t_0)}^{h_n(t_0)}\eta(u(x,t_0)|\Psi_{\bar{v}}(x,t_0))\,dx\\
&\hspace{1in}+\frac{a_1}{a_n}\int\limits_{h_n(t_0)}^{h_{\text{right}}(t_0)}\eta(u(x,t_0)|\Psi_{\bar{v}}(x,t_0))\,dx\Bigg]
\\
&\hspace{.25in}-\Bigg[\int\limits_{h_{\text{left}}(0)}^{h_1(0)}\eta(u^0(x)|\Psi_{\bar{v}}(x,0))\,dx+a_1\int\limits_{h_1(0)}^{h_n(0)}\eta(u^0(x)|\Psi_{\bar{v}}(x,0))\,dx\\
&\hspace{1in}+\frac{a_1}{a_n}\int\limits_{h_n(0)}^{h_{\text{right}}(0)}\eta(u^0(x)|\Psi_{\bar{v}}(x,0))\,dx\Bigg].
\end{aligned}
\end{equation}

Recall \eqref{dissipation_negative_claim_Riemann_problem}, \eqref{dissipation_negative_claim_Riemann_problem_n_shock}, \eqref{h_defs_RP_1} and \eqref{r_def_RP1}. In particular, note that $\dot{h}_{\text{left}}=r$ and $\dot{h}_{\text{right}}=-r$. Then, we get from \eqref{local_compatible_dissipation_calc_RP_1},

\begin{equation}
\begin{aligned}\label{almost_final_RP_lax_strong_two_shocks}
&-\int\limits_0^{t_0} \bigg[c_n \abs{\sigma^n(\bar{v}_n,\bar{v}_{n+1})-\dot{h}_n(t)}^2+c_1 \abs{\sigma^1(\bar{v}_1,\bar{v}_{2})-\dot{h}_1(t)}^2\bigg]\,dt\\
&\hspace{.5in}+
\Bigg[\int\limits_{-R-rt_0}^{h_1(0)}\eta(u^0(x)|\Psi_{\bar{v}}(x,0))\,dx+\frac{a_1}{a_n}\int\limits_{h_n(0)}^{R+rt_0}\eta(u^0(x)|\Psi_{\bar{v}}(x,0))\,dx\Bigg]
\\
&\hspace{.5in}\geq\Bigg[\int\limits_{-R}^{h_1(t_0)}\eta(u(x,t_0)|\Psi_{\bar{v}}(x,t_0))\,dx+a_1\int\limits_{h_1(t_0)}^{h_n(t_0)}\eta(u(x,t_0)|\Psi_{\bar{v}}(x,t_0))\,dx\\
&\hspace{1in}+\frac{a_1}{a_n}\int\limits_{h_n(t_0)}^{R}\eta(u(x,t_0)|\Psi_{\bar{v}}(x,t_0))\,dx\Bigg].
\end{aligned}
\end{equation}

Note that we have also used that 
\begin{align}
a_1\int\limits_{h_1(0)}^{h_n(0)}\eta(u^0(x)|\Psi_{\bar{v}}(x,0))\,dx=0
\end{align}
due to $h_1(0)=h_n(0)=0$.

Note that for $i=1,n$ the constant $c_i$ depends on $\norm{u}_{L^\infty}$, $a_i$, $\abs{\bar{v}_i-\bar{v}_{i+1}}$ (by \eqref{shock_strength_comparable_s_systems1}), and $\abs{\bar{v}_i}$, $\abs{\bar{v}_{i+1}}$.
 
Recall \Cref{entropy_relative_L2_control_system}, which says that due to the strict convexity of $\eta$, there exist  constants $c^*,c^{**}>0$ such that
\begin{align}\label{eta_locally_quadratic_RP_2_case_1}
c^*\abs{a-b}^2\leq\eta(a|b)\leq c^{**}\abs{a-b}^2,
\end{align}
for all $a,b$ in a fixed compact set. Note that $c^*,c^{**}$ depend on bounds on the second derivative of $\eta$ on the range of $a$ and $b$. 

Recall also that $a_1$ depends on $\norm{u}_{L^\infty}$, $\abs{\bar{v}_1},\abs{\bar{v}_2}$, and $\abs{\bar{v}_1-\bar{v}_2}$. Further, $a_n$ depends on $\norm{u}_{L^\infty}$, $\abs{\bar{v}_n},\abs{\bar{v}_{n+1}}$, and $\abs{\bar{v}_n-\bar{v}_{n+1}}$. Recall the relation \eqref{shock_strength_comparable_s_systems1}.

Thus, from \eqref{almost_final_RP_lax_strong_two_shocks} we can write,
\begin{align}\label{almost_final_RP_lax_strong_two_shocks_final}
\int\limits_{-R}^{R}\abs{u(x,t_0)-\Psi_{\bar{v}}(x,t_0)}^2\,dx\leq\mu_1\int\limits_{-R-rt_0}^{R+rt_0}\abs{u^0(x)-\Psi_{\bar{v}}(x,0)}^2\,dx,
\end{align}
where the constant $\mu_1>0$ depends on $\norm{u}_{L^\infty}$, $\abs{\bar{v}_i},\abs{\bar{v}_{i+1}}$, and $\abs{\bar{v}_i-\bar{v}_{i+1}}$ for both $i=1$ and $i=n$. Further, $\mu_1$ depends on bounds on the second derivative of $\eta$ on the range of $u$ and $\bar{v}$. Recall also that due to \eqref{condition_on_epsilon_0_riemann_problem}, for both $i=1$ and $i=n$, $\mu_1$ depends on $\sup\norm{\nabla \lambda_i}$ (where the supremum runs over the range of $u$ and $\bar{v}$), and  $(\lambda_{i+1}(\bar{v}_{i+1})-\lambda_i(\bar{v}_i))$.

This gives us \eqref{RP_lax_strong_stability_theorem_result}. Note $\Psi_{\bar{v}}(x,0)=\bar{v}(x,t)$.
 
We also get from \eqref{almost_final_RP_lax_strong_two_shocks},
\begin{align}
\int\limits_0^{t_0} \bigg[\abs{\sigma^n(\bar{v}_n,\bar{v}_{n+1})-\dot{h}_n(t)}^2+\abs{\sigma^1(\bar{v}_1,\bar{v}_{2})-\dot{h}_1(t)}^2\bigg]\,dt\leq\mu_2\int\limits_{-R-rt_0}^{R+rt_0}\abs{u^0(x)-\Psi_{\bar{v}}(x,0)}^2\,dx,
\end{align}
where the constant $\mu_2>0$ depends on $\norm{u}_{L^\infty}$, $\abs{\bar{v}_i},\abs{\bar{v}_{i+1}}$, and $\abs{\bar{v}_i-\bar{v}_{i+1}}$ for both $i=1$ and $i=n$. Moreover, $\mu_2$ depends on bounds on the second derivative of $\eta$ on the range of $u$ and $\bar{v}$.  Note again $\Psi_{\bar{v}}(x,0)=\bar{v}(x,t)$.
%

\end{proof}

\subsection{\Cref{RP_theorem_3}: $L^2$ Stability for the Riemann Problem with  Extremal Shocks but No Rarefactions}

When $\bar{v}$ contains no rarefactions, then we do not require \eqref{lax_strong} or \eqref{RHS_good_sign_rarefaction}. We get the following stability result,
\begin{theorem}[$L^2$ Stability for the Riemann Problem with  Extremal Shocks but No Rarefactions]
\label{RP_theorem_3}
Fix $T>0$. Assume $u,\bar{v}\in L^\infty(\mathbb{R}\times[0,T))$ are solutions to the system \eqref{system}. Assume that $u$ and $\bar{v}$ are entropic for the entropy $\eta$. Further, assume that $u$ has strong traces (\Cref{strong_trace_definition}). 

Assume also that $\bar{v}$ is a solution to the Riemann problem \eqref{Riemann_problem} and that $\bar{v}$ has the form \eqref{standard_form_solution_RP}. If $\bar{v}$ contains a 1-shock, assume the hypotheses $(\mathcal{H})$ hold. Likewise, if $\bar{v}$ contains an n-shock, assume the hypotheses $(\mathcal{H})^*$ hold. 

Assume $\bar{v}$ does not contain any rarefactions, and if $\bar{v}$ contains any shocks, they are either a 1-shock or an n-shock.

Assume the system \eqref{system} has at least two conserved quantities ($n\geq2$).

Then there exists  $\Psi_{\bar{v}}$ with Property $(\mathcal{D})$ and verifying the following stability estimate:
\begin{align}\label{RP_no_rarefaction_no_lax_strong_stability_theorem_result}
\int\limits_{-R}^{R}\abs{u(x,t_0)-\Psi_{\bar{v}}(x,t_0)}^2\,dx\leq\mu_1\int\limits_{-R-rt_0}^{R+rt_0}\abs{u^0(x)-\bar{v}(x,0)}^2\,dx,
\end{align}
for all $t_0,R>0$ verifying $t_0\in(0,R)$ and
\begin{align}\label{R_sufficiently_large_RP_theorem_statement}
R>\max_{i}\{\mbox{Lip}[h_i]\}t_0,
\end{align}
where the max runs over the i-shock families contained in $\bar{v}$ (1-shocks and/or n-shocks) and the $h_i$ are in the context of Property $(\mathcal{D})$.

Moreover,  there is $L^2$-type control on the shift functions $h_i$:
\begin{align}\label{control_shifts_two_shifts204_no_rarefactions}
\int\limits_0^{t_0} \sum_{i}\abs{\sigma^i(\bar{v}_i,\bar{v}_{i+1})-\dot{h}_i(t)}^2\,dt\leq \mu_1\int\limits_{-R-rt_0}^{R+rt_0}\abs{u^0(x)-\bar{v}(x,0)}^2\,dx,
\end{align}
where the sum runs over the i-shock families contained in $\bar{v}$ (1-shocks and/or n-shocks).

If $\bar{v}$ contains an i-shock, the constant $\mu_1>0$ depends on $\norm{u}_{L^\infty}$, $\abs{\bar{v}_i},\abs{\bar{v}_{i+1}}$, and $\abs{\bar{v}_i-\bar{v}_{i+1}}$. Further, $\mu_1$ depends on bounds on the second derivative of $\eta$ on the range of $u$ and $\bar{v}$.
\end{theorem}
\begin{proof}
There are three cases to consider:
\begin{itemize}
\item 
$\bar{v}$ contains a 1-shock and an n-shock
\item 
$\bar{v}$ contains a 1-shock but no n-shock or $\bar{v}$ contains an n-shock but no 1-shock
\item 
$\bar{v}$ does not contain any shocks
\end{itemize}

We begin with the first case,

\hfill \break
\emph{Case} $\bar{v}$ contains a 1-shock and an n-shock. 
\hfill \break

From \eqref{dissipation_negative_claim}, we have $a_1\in(0,1)$ and a shift function $h_1:[0,T)\to\mathbb{R}$ such that $h_1(0)=0$ and
\begin{equation}
\begin{aligned}\label{dissipation_negative_claim_Riemann_problem_no_rarefaction}
a_1\big(q(u_+^1;\bar{v}_2)&-\dot{h}_1(t)\eta(u_+^1|\bar{v}_2)\big)-q(u_-^1;\bar{v}_1)+\dot{h}_1(t)\eta(u_-^1|\bar{v}_1) \leq -c_1 \abs{\sigma^1(\bar{v}_1,\bar{v}_2)-\dot{h}_1(t)}^2,
\end{aligned}
\end{equation}
where $u_{\pm}^1\coloneqq u(u(h_1(t)\pm,t)$. 

Similarly, for the n-shock, we get from \eqref{dissipation_negative_claim_n}, the existence of an $a_n\in(0,1)$ and a function $h_n:[0,T)\to\mathbb{R}$ such that $h_n(0)=0$ and
\begin{equation}
\begin{aligned}\label{dissipation_negative_claim_Riemann_problem_n_shock_no_rarefaction}
&\frac{1}{a_n}\big(q(u_+^n;\bar{v}_{n+1})-\dot{h}_n(t)\eta(u_+^n|\bar{v}_{n+1})\big)-q(u_-^n;\bar{v}_n)+\dot{h}_n(t)\eta(u_-^n|\bar{v}_n) \\
&\hspace{1in}\leq -c_n \abs{\sigma^n(\bar{v}_{n},\bar{v}_{n+1})-\dot{h}_n(t)}^2,
\end{aligned}
\end{equation}
where $u_{\pm}^n\coloneqq u(u(h_n(t)\pm,t)$.  Note that by virtue of there not being any rarefactions, $\bar{v}_n=\bar{v}_2$.

Note that from \Cref{systems_entropy_dissipation_room} we know that the constant $a_i$ depends on $\norm{u}_{L^\infty}$, $\abs{\bar{v}_i},\abs{\bar{v}_{i+1}}$, and $\abs{\bar{v}_{i}-\bar{v}_{i+1}}$, for $i=1,n$. For $i=1,n$, the constant $c_i>0$ depends on $\norm{u}_{L^\infty}$, $\abs{\bar{v}_i},\abs{\bar{v}_{i+1}}$, $\abs{\bar{v}_{i}-\bar{v}_{i+1}}$, and $a_i$.

Choose 
\begin{align}\label{R_sufficiently_large_RP_no_rarefaction}
R>\max\{\mbox{Lip}[h_1],\mbox{Lip}[h_n]\}t_0.
\end{align}

Define
\begin{equation}
\begin{aligned}\label{h_defs_RP_1_no_rarefaction}
h_{\text{left}}(t)\coloneqq -R +r(t-t_0)\\
h_{\text{right}}(t)\coloneqq R -r(t-t_0),
\end{aligned}
\end{equation}
where $r>0$ satisfies
\begin{align}\label{r_def_RP1_no_rarefaction}
\abs{q(a;b)}\leq r\eta(a|b),
\end{align}
for $a,b$ within the range of $u$ and $\bar{v}$. Note that $r>0$ exists due to $\eta(a|b)$ and $q(a;b)$ both being locally quadratic in $a-b$ and $\eta$ being strictly convex.

We use \Cref{local_entropy_dissipation_rate_systems_RP} once with $X\equiv0$, $h_{\text{left}}$, and $h_1$ and with the constant function
\begin{align}
\bar{v}_1
\end{align}
playing the role of the function $\bar{u}$ in  \Cref{local_entropy_dissipation_rate_systems_RP}. Note that $h_1(t)-h_{\text{left}}(t)>0$ for all $t$ due to \eqref{R_sufficiently_large_RP_no_rarefaction}.

We also use \Cref{corollary_dissipation_RP} with $h_1$ and $h_n$ and with
\begin{align}
\bar{v_2}
\end{align}
playing the role of $\bar{u}$.  Note that we can apply \Cref{corollary_dissipation_RP} with $h_1$ and $h_n$ because by \Cref{systems_entropy_dissipation_room}, we know that for  each $t\in[0,T)$ either $\dot{h}_1(t)<\inf\lambda_1$ or $(u_+^1,u_-^1,\dot{h}_1)$ is a 1-shock   (including possibly $u^1_+=u^1_-$ and $\dot{h}_1=\lambda_1(u^1_\pm)$). Similarly, for each $t\in[0,T)$ either $\dot{h}_n(t)>\sup\lambda_n$ or $(u^n_+,u^n_-,\dot{h}_n)$ is an n-shock (including possibly $u^n_+=u^n_-$ and $\dot{h}_n=\lambda_n(u^n_\pm)$). By the hypotheses $(\mathcal{H})$ and $(\mathcal{H})^*$, the assumptions necessary to apply  \Cref{corollary_dissipation_RP} with $h_1$ and $h_n$ are satisfied. The hypotheses $(\mathcal{H})$ and $(\mathcal{H})^*$ say that the speeds of 1-shocks and n-shocks and the characteristic speeds of the 1-family and n-family  are well-separated.

Furthermore, by virtue of \Cref{systems_entropy_dissipation_room} we know that $h_1(t)\leq h_n(t)$ for all $t$. In particular, this gives \eqref{h_verify_3}.

Finally, we use \Cref{local_entropy_dissipation_rate_systems_RP} a second time, with $X\equiv0$, $h_n$, and $h_{\text{right}}$. Note that $h_{\text{right}}(t)-h_n(t)>0$ for all $t$ due to \eqref{R_sufficiently_large_RP_no_rarefaction}. The constant function
\begin{align}
\bar{v}_{n+1}
\end{align}
plays the role of $\bar{u}$.

We now take a linear combination of the two applications of \Cref{local_entropy_dissipation_rate_systems_RP} and the one application of \Cref{corollary_dissipation_RP}. Note the space derivative of constant states in $\bar{v}$ is zero. This yields,

\begin{equation}
\begin{aligned}\label{local_compatible_dissipation_calc_RP_1_no_rarefaction}
&\int\limits_{0}^{t_0} \bigg[q(u(h_{\text{left}}(t)+,t);\bar{v}_1)-\dot{h}_{\text{left}}(t)\eta(u(h_{\text{left}}(t)+,t)|\bar{v}_1)
\\
&\hspace{.5in}-\frac{a_1}{a_n}\big(q(u(h_{\text{right}}(t)-,t);\bar{v}_{n+1})-\dot{h}_{\text{right}}(t)\eta(u(h_{\text{right}}(t)-,t)|\bar{v}_{n+1})\big)
\\
&\hspace{.5in}+
\frac{a_1}{a_n}\big(q(u_+^n;\bar{v}_{n+1})-\dot{h}_n(t)\eta(u_+^n|\bar{v}_{n+1})\big)-a_1\big(q(u_-^n;\bar{v}_n)+\dot{h}_n(t)\eta(u_-^n|\bar{v}_n)\big)
\\
&\hspace{.5in}+
a_1\big(q(u_+^1;\bar{v}_2)-\dot{h}_1(t)\eta(u_+^1|\bar{v}_2)\big)-q(u_-^1;\bar{v}_1)+\dot{h}_1(t)\eta(u_-^1|\bar{v}_1) 
\bigg]\,dt
\\
&\hspace{.14in}\geq
\Bigg[\int\limits_{h_{\text{left}}(t_0)}^{h_1(t_0)}\eta(u(x,t_0)|\bar{v}_1)\,dx+a_1\int\limits_{h_1(t_0)}^{h_n(t_0)}\eta(u(x,t_0)|\bar{v}_2)\,dx+\frac{a_1}{a_n}\int\limits_{h_n(t_0)}^{h_{\text{right}}(t_0)}\eta(u(x,t_0)|\bar{v}_{n+1})\,dx\Bigg]
\\
&\hspace{.5in}-\Bigg[\int\limits_{h_{\text{left}}(0)}^{h_1(0)}\eta(u^0(x)|\bar{v}_1)\,dx+a_1\int\limits_{h_1(0)}^{h_n(0)}\eta(u^0(x)|\bar{v}_2)\,dx+\frac{a_1}{a_n}\int\limits_{h_n(0)}^{h_{\text{right}}(0)}\eta(u^0(x)|\bar{v}_{n+1})\,dx\Bigg].
\end{aligned}
\end{equation}

Recall \eqref{dissipation_negative_claim_Riemann_problem_no_rarefaction}, \eqref{dissipation_negative_claim_Riemann_problem_n_shock_no_rarefaction}, \eqref{h_defs_RP_1_no_rarefaction} and \eqref{r_def_RP1_no_rarefaction}. In particular, note that $\dot{h}_{\text{left}}=r$ and $\dot{h}_{\text{right}}=-r$. Then, we get from \eqref{local_compatible_dissipation_calc_RP_1_no_rarefaction},

\begin{equation}
\begin{aligned}\label{almost_final_RP_lax_strong_two_shocks_no_rarefaction}
&-\int\limits_0^{t_0}\Big[c_1 \abs{\sigma^1(\bar{v}_1,\bar{v}_2)-\dot{h}_1(t)}^2+c_n \abs{\sigma^n(\bar{v}_{n},\bar{v}_{n+1})-\dot{h}_n(t)}^2 \Big]\,dt 
\\
&\hspace{.3in}+\Bigg[\int\limits_{-R-rt_0}^{h_1(0)}\eta(u^0(x)|\bar{v}_1)\,dx+\frac{a_1}{a_n}\int\limits_{h_n(0)}^{R+rt_0}\eta(u^0(x)|\bar{v}_{n+1})\,dx\Bigg]
\\
&\hspace{.3in}\geq\Bigg[\int\limits_{-R}^{h_1(t_0)}\eta(u(x,t_0)|\bar{v}_1)\,dx+a_1\int\limits_{h_1(t_0)}^{h_n(t_0)}\eta(u(x,t_0)|\bar{v}_2)\,dx+\frac{a_1}{a_n}\int\limits_{h_n(t_0)}^{R}\eta(u(x,t_0)|\bar{v}_{n+1})\,dx\Bigg].
\end{aligned}
\end{equation}

Note that we have also used that 
\begin{align}
a_1\int\limits_{h_1(0)}^{h_n(0)}\eta(u^0(x)|\bar{v}_2)\,dx=0
\end{align}
due to $h_1(0)=h_n(0)=0$.
 
The reader will recall \Cref{entropy_relative_L2_control_system}: by virtue of the strict convexity of $\eta$, there exist  constants $c^*,c^{**}>0$ such that
\begin{align}\label{eta_locally_quadratic_RP_2_case_1_no_rarefaction}
c^*\abs{a-b}^2\leq\eta(a|b)\leq c^{**}\abs{a-b}^2,
\end{align}
for all $a,b$ in a fixed compact set. Note that $c^*,c^{**}$ depend on bounds on the second derivative of $\eta$ on the range of $a$ and $b$. 

Recall also that $a_1$ depends on $\norm{u}_{L^\infty}$, $\abs{\bar{v}_1},\abs{\bar{v}_2}$, and $\abs{\bar{v}_1-\bar{v}_2}$ (via the relation \eqref{shock_strength_comparable_s_systems1}). Similarly, $a_n$ depends on $\norm{u}_{L^\infty}$, $\abs{\bar{v}_n},\abs{\bar{v}_{n+1}}$, and $\abs{\bar{v}_n-\bar{v}_{n+1}}$ (via the relation \eqref{shock_strength_comparable_s_systems1}). 

Thus, from \eqref{almost_final_RP_lax_strong_two_shocks_no_rarefaction} we can write,
\begin{align}\label{almost_final_RP_lax_strong_two_shocks_final_no_rarefaction}
\int\limits_{-R}^{R}\abs{u(x,t_0)-\Psi_{\bar{v}}(x,t_0)}^2\,dx\leq\mu_1\int\limits_{-R-rt_0}^{R+rt_0}\abs{u^0(x)-\Psi_{\bar{v}}(x,0)}^2\,dx,
\end{align}
for a constant $\mu_1>0$. This gives us \eqref{RP_no_rarefaction_no_lax_strong_stability_theorem_result}. Note $\Psi_{\bar{v}}(x,0)=\bar{v}(x,0)$. 
 
From \eqref{almost_final_RP_lax_strong_two_shocks_no_rarefaction}, we also have that
\begin{align}
\int\limits_0^{t_0}\Bigg[ \abs{\sigma^1(\bar{v}_1,\bar{v}_2)-\dot{h}_1(t)}^2+ \abs{\sigma^n(\bar{v}_{n},\bar{v}_{n+1})-\dot{h}_n(t)}^2 \Bigg]\,dt \leq\mu_1\int\limits_{-R-rt_0}^{R+rt_0}\abs{u^0(x)-\Psi_{\bar{v}}(x,0)}^2\,dx.
\end{align}

This is \eqref{control_shifts_two_shifts204_no_rarefactions}. Note again $\Psi_{\bar{v}}(x,0)=\bar{v}(x,0)$. 
 
Note that the constant $\mu_1>0$ depends on $\norm{u}_{L^\infty}$, $\abs{\bar{v}_i},\abs{\bar{v}_{i+1}}$, and $\abs{\bar{v}_i-\bar{v}_{i+1}}$ for both $i=1$ and $i=n$. Further, $\mu_1$ depends on bounds on the second derivative of $\eta$ on the range of $u$ and $\bar{v}$.

\hfill \break
\emph{Case} $\bar{v}$ contains a 1-shock but no n-shock or $\bar{v}$ contains an n-shock but no 1-shock
\hfill \break

This case is very similar to the above case. In fact, it is simpler because we do not need to use \Cref{corollary_dissipation_RP}. We can simply use \Cref{local_entropy_dissipation_rate_systems_RP}.

\hfill \break
\emph{Case} $\bar{v}$ does not contain any shocks
\hfill \break

In this case, $\bar{v}$ does not contain shocks or rarefactions. Thus $\bar{v}$ is a constant function. Then, \eqref{RP_no_rarefaction_no_lax_strong_stability_theorem_result} follows from the classical weak-strong stability theorem (see \cite[Theorem 5.2.1]{dafermos_big_book}).

\end{proof}

\bibliographystyle{plain}
\bibliography{references}
\end{document}